\documentclass[a4paper,11 pt]{article}

\usepackage[utf8]{inputenc}

\usepackage[french]{}
\usepackage{newlfont}
\usepackage{amsmath, amscd}
\usepackage{amsthm}
\usepackage{amssymb}
\usepackage{booktabs}
\usepackage[a4paper,top=3.5cm,bottom=3.5cm,left=3.5cm,right=3.5cm,heightrounded]{geometry}
\usepackage[font=small,labelfont={bf}]{caption}
\usepackage{subfig}
\usepackage{caption}
\usepackage{graphicx}
\usepackage[T1]{fontenc}
\usepackage{listings}
\usepackage{setspace}
\usepackage{mathtools}
\usepackage{textcomp}
\usepackage{array}
\usepackage{multicol}

\title{Teleman's classification of semisimple cohomological field theories}
\author{Simone Melchiorre Chiarello\footnote{Simone.Chiarello@unige.ch, Université de Genève, Villa Battelle, route de Drize 9, 1227 Carouge (Switzerland). Partially supported by the ENS grant \textquotedblleft Sélection Internationale\textquotedblright.}}
\date{}

\newtheorem{theorem}{Theorem}
\newtheorem{lemma}{Lemma}
\newtheorem{corollary}{Corollary}
\newtheorem{proposition}{Proposition}

\theoremstyle{definition}
\newtheorem{definition}{Definition}
\newtheorem{remark}{Remark}
\newtheorem{notation}{Notation}
\newtheorem{example}{Example}

\begin{document}

\maketitle

\begin{abstract}
 In his paper \cite{teleman:articolo} Teleman proved that a cohomological field theory on the moduli space 
 $\overline{\mathcal{M}}_{g,n}$ of stable complex curves
 is uniquely determined by its restriction to the smooth part $\mathcal{M}_{g,n}$,
 provided that the underlying Frobenius algebra is semisimple. This leads to a classification of all
 semisimple cohomological field theories. The present paper, the outcome of the author's 
 master's thesis, presents Teleman's proof
 following the above-mentioned paper. The author claims no originality: the main motivation has been to keep the
 exposition as complete and self-contained as possible.
\end{abstract}

\section{Introduction}
In this work, we will classify the
 semisimple cohomological field theories (CohFT), following 
 Teleman's 2011 article \cite{teleman:articolo}. A cohomological field theory is the datum of a finite dimensional
 $\mathbb{C}$-vector space $A$ with a nondegenerate bilinear 
 form $\eta$ on it, a nonzero vector $\mathbf{1}\in A$,
 and for every $g$ and $n$ with
 $2g-2+n > 0$, a $\mathbb{C}$-linear homomorphism
 $$\overline{\Omega}_{g,n}: A^{\otimes n} \rightarrow H^{\bullet}(\overline{\mathcal{M}}_{g,n})$$
 satisfying certain axioms related to the geometrical structure of $\overline{\mathcal{M}}_{g,n}$ (the moduli
 space of stable complex projective curves with $n$ marked points), namely the presence of a natural $S_n$-action, of 
 sewing maps and of the forgetful map $\overline{\mathcal{M}}_{g,n} \rightarrow \overline{\mathcal{M}}_{g,n-1}$. More
 details will be given in Section 3.
 
 We define a Frobenius algebra structure on $A$ by imposing 
 $\overline{\Omega}_{0,3}(v_1 \otimes v_2 \otimes v_3) = \eta(v_1 \cdot v_2, v_3)$,
 so we get $\mathbf{1}$ as the neutral element for the multiplication.
 This algebra is said to be \emph{semisimple} if there exists an orthonormal basis $P_i$ for $\eta$ made of
 projectors, that is elements such that $P_i \cdot P_i = \theta_i^{-1} P_i$ for some
 nonzero complex numbers $\theta_i^{-1}$ and $P_i \cdot P_j = 0$ for $i \neq j$; in this
 case the CohFT is said to be semisimple. The classification will be restricted to this type of theories.
 One property of semisimple Frobenius algebras, that will be used in a crucial way, is that 
 the element $\alpha = \sum_i \theta_i^{-1}P_i$, called the \emph{Euler class} of the algebra, is invertible. 
 
 The classification will be carried out in three steps. First of all, we will consider the space 
 $\widetilde{\mathcal{M}}_{g,n}$ which is
 the $n$-torus bundle over $\mathcal{M}_{g,n}$ corresponding to assigning a
 tangent direction to the curve at each marked point, and we define
 $$\widetilde{\Omega}_{g,n}: A^{\otimes n} \rightarrow H^{\bullet}(\widetilde{\mathcal{M}}_{g,n})$$
 with the same axioms of a cohomological field theory, properly modified with a deformation of the sewing maps
 so to always have smooth curves. These theories are called \emph{fixed boundary theories} 
 and can be easily classified thanks to several fundamental results concerning the cohomology
 of the smooth part $\mathcal{M}_{g,n}$, namely Harer's stability theorem for $\mathcal{M}_g$,
 its extension to $\mathcal{M}_{g,n}$ by Looijenga,
 and Madsen-Weiss's theorem (former Mumford's conjecture) about the stable
 cohomology of $\mathcal{M}_g$.
 We will see that the answer for the first step is given by a homomorphism 
 $\widetilde{\Omega}^+ : A \rightarrow \mathbb{C}[\kappa_j]_{j\ge 1}$ of the form 
 $\widetilde{\Omega}^+ = \text{exp}(\sum_{j>0} \phi_j \kappa_j)$ 
 where the $\phi_j \in A^*$ are freely chosen co-vectors,
 and the $\kappa_j$'s are the $\kappa$-classes on the moduli space. For every choice of the $\phi_j$'s we have
 a theory by setting $\widetilde{\Omega}_{g,n}(v_1 \otimes \ldots \otimes v_n) = 
 \widetilde{\Omega}^+(\alpha^g \cdot v_1 \cdot \ldots \cdot v_n)|_{\widetilde{\mathcal{M}}_{g,n}}$.
 
 The second step deals with CohFTs on $\mathcal{M}_{g,n}$, again with axioms
 slightly modified to always have smooth curves. These are called \emph{free boundary theories}
 and in this case we get the classification by pulling back
 to $\widetilde{\mathcal{M}}_{g,n}$ (this procedure yields indeed a fixed boundary theory)
 and using the result of the first step. The classification in this case is 
 given by a fixed boundary theory and a $\psi$-valued endomorphism
 $R(\psi) \in \text{End}(A)[[\psi]]$ with $R(0) = \text{Id}$, satisfying the 
 \emph{symplectic condition} $$R(\psi)^* = R(-\psi)^{-1}$$ and such that
 $$ \text{log}\widetilde{\Omega}^+(v) = - \eta(\beta\text{log}(R(\psi)^{-1}\mathbf{1}),v) .$$
 The quantities involved are defined in Section 6.2.
 
 Under these conditions, the theory is given by
 $$\Omega_{g,n}(v_1 \otimes \ldots \otimes v_n) = 
 \widetilde{\Omega}^+(\alpha^g \cdot R(\psi_1)^{-1}(v_1) 
 \cdot \ldots \cdot R(\psi_n)^{-1}(v_n))|_{\mathcal{M}_{g,n}}$$
 where the $\psi_i$ are the $\psi$-classes at the marked points.

 The last step is to get
 the full theories on $\overline{\mathcal{M}}_{g,n}$, also called 
 Deligne-Mumford theories or nodal theories. The fact that for every smooth theory $\Omega_{g,n}$
 there is a nodal theory $\overline{\mathcal{M}}_{g,n}$ that restricts to $\Omega_{g,n}$ on the smooth part
 $\mathcal{M}_{g,n}$ will be an easy consequence of Givental's group action. However, the fact that 
 this nodal theory is \emph{unique} is not trivial at all. To show it, a new stratification of 
 $\overline{\mathcal{M}}_{g,n}$ is introduced.

 \subsection*{Acknowledgments}
 
 This work has been written on spring 2015 as a master's thesis at the Université Pierre et Marie Curie (Paris VI)
 under the supervision of Dimitri Zvonkine: I would like to thank him for his infinite patience and for having taught me
 all I know about this subject and about how to write a paper. I also thank the jury members Alessandro Chiodo and
 Julien Marché. I thank Alessandro Chiodo for very useful discussions and support about 
 my future as a PhD student. I would also like to thank the ENS for having provided
 the funds and the resources that have allowed me to complete my career as a master student.
 
\section{Basic notions}

We collect in this section all the basic notions that will be assumed throughout the paper. This is just
an expository section and contains no proofs. A good introduction can be found, for example, 
in \cite{arbarello:curves} or in \cite{zvonkine:notemoduli}. For the founding principles of moduli spaces,
see Mumford's book \cite{mumford:git}.

\subsection{Moduli spaces of stable curves}

A a Riemann surface $C$ with $n$ distinct marked points $x_1, \ldots, x_n$ is \emph{stable} if there is
only a finite number of automorphisms of $C$ that leave all the points $x_i$ fixed. It can be shown that
a Riemann surface of genus $g$ with $n$ marked point is stable if and only if $2g-2+n > 0$.
We will use the notation $\mathcal{M}_{g,n}$ to denote the coarse
moduli space of stable Riemann surfaces of genus $g$ and
$n$ marked points; roughly speaking, it is the class of all stable Riemann surfaces of genus $g$ with $n$ marked
points, modulo isomorphism. For example, $\mathcal{M}_{0,3}$ consists of just one point: any Riemann surface
of genus $0$ is isomorphic to $\mathbb{P}^1(\mathbb{C})$, and for any two sets consisting each of three distinct
points, there is an automorphism of $\mathbb{P}^1(\mathbb{C})$ that sends one set into the other one.
We will denote by $(C, x_1, \ldots, x_n) \in \mathcal{M}_{g,n}$ a stable curve, where
$C$ is a Riemann surface of genus $g$ and $x_1, \ldots, x_n \in C$. The space $\mathcal{M}_{g,n}$ exists only
in the case that $2g-2+n>0$; in this case, its complex dimension is $3g - 3 + n$.

We will write $\overline{\mathcal{M}}_{g,n}$ for the the \emph{Deligne-Mumford compactification} of
$\mathcal{M}_{g,n}$. It is a compact orbifold that has $\mathcal{M}_{g,n}$ as an open dense suborbifold.
The elements of $\overline{\mathcal{M}}_{g,n}$ are called \emph{stable curves} of genus $g$ with $n$ marked points.
It can be shown that any stable curve is a connected nodal curve, with only simple nodes.

For $(C,x_1, \ldots, x_n)$ a nodal curve of geometric genus $g$ with only simple nodes and $n$ marked points
distinct from the nodes, let $\nu : C' \rightarrow C$ be its normalization. Then each node $y_i$ has two
pre-images $y_i'$ and $y_i''$ on $C'$, while each marked point $x_i$ has one pre-image $x_i' \in C'$.
If $C = C_1 \cup \ldots \cup C_k$ where the $C_i$'s are the irreducible components of $C$, we have a corresponding
decomposition $C' = C_1' \sqcup \ldots \sqcup C_k'$ into \emph{connected} smooth components, and for each $i$ we have
$\nu|_{C_i}: C'_i \twoheadrightarrow C_i$.

Let $(C,x_1, \ldots, x_n)$ be a connected nodal curve with only simple nodes which are distinct from the marked points,
let $N \subseteq C$ be the set of its nodes, and $M = \{ x_1, \ldots x_n \}$ be the set of its marked points.
Each component $C_i'$ has natural markings, namely the $k'$ points of
$N_i' = \nu^{-1}(N)\cap C_i'$ and the $n'$ points
of $M_i' = \nu^{-1}(M) \cap C_i'$. Then $(C,x_1, \ldots, x_n)$ is stable 
(that is, it belongs to $\overline{\mathcal{M}}_{g,n}$) if and only if every $(C_i, N_i' \sqcup M_i')$ is a
stable Riemann surface, that is, it belongs to $\mathcal{M}_{g_i, n_i}$ for some $g_i$ and $n_i$ such that
$2g_i - 2 + n_i > 0$.
A point which is either a node or a marked point is called a \emph{special point}.

If $(C_1, x_1, \ldots, x_{n_1+1}) \in \overline{\mathcal{M}}_{g_1,n_1+1}$ 
and $(C_2, y_1, \ldots, y_{n_2+1}) \in \overline{\mathcal{M}}_{g_2,n_2+1}$, let 
$$s(C_1, C_2) = (C_1 \sqcup C_2)/ \sim $$
where $\sim$ is the equivalence relation in which $x_{n_1+1} \sim y_{n_2+1}$ and all the other points are 
 identified with nothing but themselves. Then $s(C_1, C_2)$ is a nodal curve and it can be easily shown that
 $(s(C_1,C_2), x_1, \ldots, x_{n_1}, y_1, \ldots y_{n_2})$ is in fact a stable curve of genus $g = g_1 + g_2$
 with $n = n_1 + n_2$ marked points. Therefore, $s$ induces a continuous map
 $$s: \overline{\mathcal{M}}_{g_1, n_1+1} \times \overline{\mathcal{M}}_{g_2,n_2+1} 
 \rightarrow \overline{\mathcal{M}}_{g,n}$$
which we will call the \emph{separating sewing map}, and the operation of taking $s$ will be also described as
\textquotedblleft sewing the curves in $\overline{\mathcal{M}}_{g_1, n_1+1}$ with those in 
$\overline{\mathcal{M}}_{g_2,n_2+1}$ along the points marked $n_1+1$ and $n_2+1$ respectively\textquotedblright.
Notice that $s$ is an injective map.

Analogously, if $(C, x_1, \ldots, x_{n+2}) \in \overline{\mathcal{M}}_{g-1,n+2}$, let
$$q(C) = C/\sim$$
where $\sim$ is the equivalence relation in which $x_{n+1} \sim x_{n+2}$ and all the other points are identified
with nothing but themselves. Then $q(C)$ is a nodal curve and it can be easily shown that
$(q(C), x_1, \ldots, x_n)$ is in fact a stable curve of genus $g$ with $n$ marked points. Therefore, $q$ induces 
a continuous map
$$q: \overline{\mathcal{M}}_{g-1,n+1} \rightarrow \overline{\mathcal{M}}_{g,n}$$
which we will call the \emph{non-separating sewing map}. Notice that $q$ is a $2$-fold covering map on its image.

Let us suppose that $2g-2+n > 0$ and let $(C, x_1, \ldots, x_{n+1}) \in \overline{\mathcal{M}}_{g,n+1}$.
If the curve $(C,x_1, \ldots, x_n)$ obtained
by cancelling the last marked point is stable, we define it to be $p(C, x_1, \ldots, x_{n+1})$.
However, $(C,x_1, \ldots, x_n)$ can also be not stable. This happens if the last marked point lies
on an irreducible component that has genus zero and only three special points. If this is the case, we call
the \emph{stabilization} of $(C,x_1, \ldots, x_n)$ the curve obtained by contracting the whole non-stable component
to a point. We define this stabilization to be $p(C,x_1, \ldots, x_n, x_{n+1})$. Therefore, $p$ induces a continuous
map
$$p: \overline{\mathcal{M}}_{g,n+1} \rightarrow \overline{\mathcal{M}}_{g,n}$$
called the \emph{forgetful map}, and the operation of taking $p$ will be also described as
\textquotedblleft forgetting the last marked point\textquotedblright.
Notice that $p$ is a proper map and it is surjective.

\subsubsection{Cohomology classes on $\overline{\mathcal{M}}_{g,n}$}

We denote by $H^{\bullet}(\overline{\mathcal{M}}_{g,n})$ the cohomology with complex coefficients of 
$\overline{\mathcal{M}}_{g,n}$ considered just as a topological space.
We use a similar notation for $\mathcal{M}_{g,n}$.

For each $(C, x_1, \ldots, x_n) \in \overline{\mathcal{M}}_{g,n}$, let $(T_{x_i}C)^*$ be the cotangent space 
to $C$ at $x_i$. Assigning this object to each curve of $\overline{\mathcal{M}}_{g,n}$ defines a line bundle
$$L_i \rightarrow \overline{\mathcal{M}}_{g,n}$$
which we will call, by abuse of language, the \textquotedblleft cotangent bundle at
the $i$-th marked point\textquotedblright.
Its Chern class is denoted by 
$$\psi_i := \text{c}(L_i) \in H^2(\overline{\mathcal{M}}_{g,n})$$
and we will call it the $\psi$-class at the $i$-th marked point.

Let us consider the forgetful map $p: \overline{\mathcal{M}}_{g,n+1} \rightarrow \overline{\mathcal{M}}_{g,n}$. 
Since it is a proper map, the push-forward $p_* : H^{\bullet}(\overline{\mathcal{M}}_{g,n+1}) \rightarrow
H^{\bullet-2}(\overline{\mathcal{M}}_{g,n})$ is well-defined. We define
$$\kappa_m = p_*(\psi_{n+1}^{m+1}) \in H^{2m}(\overline{\mathcal{M}}_{g,n})$$
and we will call it the $m$-th $\kappa$-class.

\subsubsection{Boundary strata}
Let us fix an element $C \in \overline{\mathcal{M}}_{g,n}$, and let $S_C \subseteq \overline{\mathcal{M}}_{g,n}$
be the set of curves which are diffeomorphic to $C$. Then $S_C$ is, by definition, the \emph{open boundary stratum}
corresponding to the topological type of $C$.
The topological closure $\overline{S_C}$ of an open boudary stratum is called a \emph{closed boundary stratum}.

If $C$ is a smooth curve, we get $S_C = \mathcal{M}_{g,n}$, therefore $\mathcal{M}_{g,n}$ is an open boundary
stratum, and we will call it the \emph{smooth stratum} of $\overline{\mathcal{M}}_{g,n}$.
By contrast, the complementary
$\overline{\mathcal{M}}_{g,n} \setminus \mathcal{M}_{g,n} = \partial \overline{\mathcal{M}}_{g,n}$ is called
the \emph{boundary} of $\overline{\mathcal{M}}_{g,n}$. Except for the smooth stratum, every open or closed boundary
stratum is a subset of the boudary $\partial \overline{\mathcal{M}}_{g,n}$ (this justifies 
the term \textquotedblleft boundary stratum\textquotedblright). Notice that, while a closed boundary stratum
is obviously a closed subset of $\overline{\mathcal{M}}_{g,n}$, an open boundary stratum which is not the
smooth stratum is \emph{never} an open subset of $\overline{\mathcal{M}}_{g,n}$. It can be shown that
open boundary strata are smooth suborbifolds of $\overline{\mathcal{M}}_{g,n}$, meaning that they have
no self intersections.

Let $C \in \overline{\mathcal{M}}_{g,n}$ and let $\nu : C' \rightarrow C$ be its normalization, with
$C' = C_1' \cup \ldots \cup C_k'$ the connected components, and let $C_i' \in \overline{\mathcal{M}}_{g_i,n_i}$.
Then it is clear
(one can prove it by induction on $k$) that $\nu$ is simply a sequence of
$s$ and $q$ maps that can be taken, without loss of generality, in order:
\begin{equation}
\label{sequenceofsewings}
\nu = q_1 \circ \ldots \circ q_l \circ s_1 \circ \ldots \circ s_m.
\end{equation}
Since each $s_i$ is injective and each $q_j$ is a covering map, we see that the sequence written above
induces a covering map from $\prod_{i=1}^k\mathcal{M}_{g_i,n_i}$ to $S_C$; let $F$ be the monodromy
group of this covering. Then we have
$$S_C \simeq (\prod_{i=1}^k\mathcal{M}_{g_i,n_i})/F.$$
Therefore we also get $\overline{S_C} \simeq (\prod_{i=1}^k\overline{\mathcal{M}}_{g_i,n_i})/F$.
From this description, it is clear that
the correspondence $S_C \mapsto \overline{S_C}$ is a bijection between the open and the closed boundary strata.

Let $N$ be the normal bundle of $S_C$ in $\overline{\mathcal{M}}_{g,n}$, and let $(x_i', x_i'')$ for
$i=1,\ldots, l+m$ be the pairs of points that are sewed together by the maps in the sequence (\ref{sequenceofsewings}).
Then it can be shown that 
$$N \simeq \bigoplus_{i=1}^{l+m}T_{x_i'}\otimes T_{x_i''}$$
where $T_{x_i'}$ and $T_{x_i''}$ are, respectively, the tangent bundles at the marked points $x_i'$ and $x_i''$.
Therefore, by definition of $\psi$-classes, its Chern class is
$$\text{c}(N) = -\sum_{i=1}^{l+m}(\psi_i' + \psi_i'').$$

\subsubsection{The stable range}

The \emph{stable range} of the cohomology of $\overline{\mathcal{M}}_{g,n}$ is the part of
$H^{\bullet}(\overline{\mathcal{M}}_{g,n})$ of degree at most $g/3$ (unless when the converse is clear,
all degrees, as long as dimensions and codimensions, are intended as complex). The reason for
this terminology will be more clear when we will state Harer's stability theorem in the next section.

In the stable range we have the following crucial theorems.

\begin{theorem}[Looijenga]
 Let $g \ge 2$ and $\mathcal{M}_g = \mathcal{M}_{g,0}$. Then if $d \le g/3$ we have 
 $$H^d(\mathcal{M}_{g,n}) \cong (H^{\bullet}(\mathcal{M}_g)[\psi_1, \ldots, \psi_n])_d$$
 where the subscript $d$ at the right-hand side indicates the component of degree $d$.
\end{theorem}

\begin{proof}
 See, for example, Corollary 2.18 of \cite{teleman:articolo}. For the improvement of the bound on the degree, see 
 \cite{wahl:articolo}.
\end{proof}

The next theorem has been known as \textquotedblleft Mumford's conjecture\textquotedblright \ for a long time.

\begin{theorem}[Madsen-Weiss, 2005]
 Let $g \ge 2$ and let $\mathcal{M}_g = \mathcal{M}_{g,0}$. Then if $d \le g/3$ we have
 $$H^d(\mathcal{M}_g) \cong (\mathbb{C}[\kappa_j]_{j \ge 1})_d$$
 where the subscript $d$ at the right-hand side indicates the component of degree $d$.
\end{theorem}

\begin{proof}
 See \cite{madsenweiss:limit}.
\end{proof}

\subsection{The tubular neighbourhood theorem}

Let $M$ be a smooth manifold and $X \subseteq M$ a closed submanifold of codimension $k$.
By \emph{tubular neighbourhood} of $X$ in $M$ we mean an open neighbourhood $T$ of $X$ in $M$,
endowed with a smooth map $\pi:T \rightarrow X$ with fibers homeomorphic to $B^k$,
the open real ball of dimension $k$.
We state the following technical theorem, which will be crucial in the sequel.

\begin{theorem}[Tubular neighbourhood theorem]
 Let $M$ be a smooth manifold and $X \subseteq M$ a closed submanifold. Let $\nu: N \rightarrow X$ be the
 normal bundle of $X$ in $M$. Then there exists a tubular neighbourhood $T \subseteq M$ of $X$ and an homeomorphism
 $\phi: N \xrightarrow{\sim} T$ such that $\pi \circ \phi = \nu$. Both $T$ and $\phi$ are unique, up to homotopy.
\end{theorem}

\begin{proof}
 See Theorem 6.5 in \cite{ana:symplectic}.
\end{proof}

 \section{Cohomological field theories}

(Our definition and treatment of cohomological field theories will be in Kontsevich-Manin's style,
as in \cite{kontmanin:gromwit}. We will not use universal classes as in \cite{teleman:articolo}, but
it can be seen that the two definitions are equivalent.)

Let 
$$q: \overline{\mathcal{M}}_{g-1,n+2} \rightarrow \overline{\mathcal{M}}_{g,n}$$
be the map that sews the points labeled $n+1$ and $n+2$, and let
$$s: \overline{\mathcal{M}}_{g_1, n_1 + 1} \times \overline{\mathcal{M}}_{g_2,n_2+1} \rightarrow
\overline{\mathcal{M}}_{g,n} $$ 
be the map that sews together two surfaces at the points marked $n_1+1$ and $n_2 + 1$ respectively; here
$g = g_1 + g_2$ and $n = n_1 + n_2$.
Finally let 
$$p: \overline{\mathcal{M}}_{g,n+1}\rightarrow \overline{\mathcal{M}}_{g,n}$$
be the map that forgets the marked point labeled $n+1$ and then stabilizes the resulting curve.

\subsection{Nodal theories}

We start from nodal theories since they have the simplest sewing axioms.

\begin{definition}
\label{nodaltheories}
Let $A$ be a finite-dimensional $\mathbb{C}$-vector space, $\mathbf{1} \in A$ a non-zero vector,
and $\eta = \eta_{\mu \nu}e^{\mu}\otimes e^{\nu}$ 
a nondegenerate symmetric bilinear form on $A$ with inverse bi-vector $\eta^{\mu \nu} e_{\mu} \otimes e_{\nu}$.
A \emph{nodal cohomological field theory} (nodal CohFT) with base $(A,\eta, \mathbf{1})$
is a set of linear maps

$$\overline{\Omega}_{g,n}: A^{\otimes n} \rightarrow H^{\bullet}(\overline{\mathcal{M}}_{g,n})$$
for every nonnegative
integer numbers $n$, $g$ such that $2g - 2 + n > 0$, that satisfy the following conditions:

\begin{enumerate}
 \item $\overline{\Omega}_{g,n}$ is $S_n$-equivariant for every $g$, $n$;
 \item $\overline{\Omega}_{0,3}(\mathbf{1} \otimes u \otimes v) = \eta(u,v)$, for every $u,v \in A$;
 \item $q^{*}\overline{\Omega}_{g,n}(v_1 \otimes \ldots \otimes v_n) = 
 \eta^{\mu \nu}\overline{\Omega}_{g-1, n+2}(v_1 \otimes \ldots \otimes v_n \otimes e_{\mu} \otimes e_{\nu})$,
 for any $v_1, \ldots, v_n \in A$;
 \item $s^{*}\overline{\Omega}_{g,n}(v_1 \otimes \ldots \otimes v_n) = 
 \eta^{\mu \nu} \overline{\Omega}_{g_1, n_1 + 1}(v_1 \otimes \ldots \otimes v_{n_1} \otimes e_{\mu}) \times
 \overline{\Omega}_{g_2, n_2 + 1}( v_{n_1 + 1} \otimes \ldots \otimes v_n \otimes e_{\nu})$,
 for any $v_1, \ldots, v_n \in A$;
 \item $p^{*}\overline{\Omega}_{g,n}(v_1 \otimes \ldots \otimes v_n) 
 = \overline{\Omega}_{g,n+1}(v_1 \otimes \ldots \otimes v_n
 \otimes \mathbf{1})$,
 for any $v_1, \ldots, v_n \in A$.
\end{enumerate}
The summation over $\mu$ and $\nu$ is always assumed.

\end{definition}

\begin{remark}
\label{equivariance}
We explain better the meaning of $S_n$-equivariance. Let $\rho \in S_n$ be a permutation of the
index set $\{ 1, \ldots, n\}$. Since 
we have an action of $S_n$ on $\overline{\mathcal{M}}_{g,n}$ by permutation of the marked points, we have an 
induced cohomology map $\rho^{*}: H^{\bullet}(\overline{\mathcal{M}}_{g,n})
\rightarrow H^{\bullet}(\overline{\mathcal{M}}_{g,n})$. The first axiom then means that for every $\rho \in S_n$,
we must have
$$ \overline{\Omega}_{g,n}(v_{\rho(1)} \otimes \ldots \otimes v_{\rho(n)}) = 
\rho^{*}\overline{\Omega}_{g,n}(v_1 \otimes \ldots \otimes v_n)$$
\end{remark}

\subsection{Fixed boundary theories}

We now define fixed boundary theories. They are smooth theories, meaning that they deal with the moduli space of
smooth Riemann surfaces.

\begin{notation}
\label{torusbundle}

For every $g$ and $n$ such that $2g-2+n>0$, 
let $I \subseteq \{1, \ldots, n \}$ and let
$$\pi_I : \widetilde{\overline{\mathcal{M}}^I}_{g,n} \rightarrow \overline{\mathcal{M}}_{g,n}$$
be the torus bundle defined as follows.

For a complex vector space $V$, we call $S(V)$ its
\emph{spherization}, defined by 
$$S(V) = (V\setminus \{0\})/{\mathbb{R}^+}.$$
 If $C \in \overline{\mathcal{M}}_{g,n}$
with marked points $x_1, \ldots, x_n$,
then the fiber of $\pi_I$ at $C$ is $\prod_{i \in I}S(T_{x_i}C)$ where
each $T_{x_i}C$ is the complex tangent line to $C$
at the point $x_i$.  

If $I = \{ 1, \ldots, n \}$ we will simply write $\widetilde{\overline{\mathcal{M}}}_{g,n}$ for
$\widetilde{\overline{\mathcal{M}}^I}_{g,n}$ and $\pi$ for $\pi_I$.

We also write $\pi^I: \widetilde{\overline{\mathcal{M}}}_{g,n} 
\rightarrow \widetilde{\overline{\mathcal{M}}^I}_{g,n}$
for the $\mathbb{T}^{n-|I|}$-bundle whose fiber at each point of $\pi_I^{-1}(C)$ is
$\prod_{i \notin I} S(T_{x_i}C)$.

We write $\widetilde{\mathcal{M}}_{g,n}^I$ for the restriction of $\pi_I$ to $\mathcal{M}_{g,n}$.
\end{notation}

\begin{remark}
\label{action}
Each $S(T_{x_i}C)$ is a circle bundle endowed with a natural action of $S^1$, the
set of complex numbers with norm $1$, defined by 
$$z \cdot [v] = [zv]$$
for $z \in S^1$ and $[v] \in S(T_{x_i}C)$.
Let $I = \{i_1, \ldots, i_{|I|}\}$, let
$(C,x_1, \ldots x_n) \in \mathcal{M}_{g,n}$ and $[v_{i_j}] \in S(T_{x_{i_j}}C)$. We 
have an action of $\mathbb{T}^{|I|}=(S^1)^{|I|}$ on $\widetilde{\mathcal{M}}^I_{g,n}$
defined by
$$ (z_1, \ldots, z_{|I|}) \cdot (C, [v_{i_1}], \ldots, [v_{i_{|I|}}]) =
(C,[z_1v_{i_1}],\ldots,[z_{|I|}v_{i_{|I|}}]).$$
Thus $\pi_I$ is a principal $\mathbb{T}^{|I|}$-bundle.
\end{remark}

\begin{remark} It is clear that this bundle can be thought 
of as a way of \textquotedblleft attaching a tangent direction to each marked point\textquotedblright , thus
$\widetilde{\mathcal{M}}_{g,n}$ is the moduli space
of Riemann surfaces with marked points \emph{and} tangent directions
for each of them.
\end{remark}

\begin{definition}
If $I \subseteq \{ 1, \ldots, n\}$ and $i \in I$, the point $x_i$ is a \emph{framed point} of
$\widetilde{\mathcal{M}}_{g,n}^I$; if $i \notin I$,
the point $x_i$ is a \emph{free point}.
\end{definition}

Let $I_1 = \{ 1, 2, \ldots, n_1 \}$ and $I_2 = \{1, 2, \ldots, n_2 \}$ and let
$$\tilde{s} : \widetilde{\mathcal{M}}^{I_1}_{g_1,n_1+1}
\times \widetilde{\mathcal{M}}^{I_2}_{g_2, n_2+1} \rightarrow 
\widetilde{\overline{\mathcal{M}}}_{g,n}$$
be the map that sews the points marked $n_1 + 1$ and $n_2+1$ respectively on the first and second curve. Note that
these points are free, 
so that we do not have to worry about what it means sewing two framed points yet.
Let $\widetilde{S}$ be the image of $\tilde{s}$.

\begin{remark}
\label{codim1}
This image
$\widetilde{S} \subseteq \widetilde{\overline{\mathcal{M}}}$ is a boundary stratum of codimension $1$.
In fact, if $s$ is the usual sewing map,
we can draw the following commutative diagram:

\[
\begin{CD}
\widetilde{\mathcal{M}}^{I_1}_{g_1,n_1+1} \times \widetilde{\mathcal{M}}^{I_2}_{g_2, n_2+1} 
@>\tilde{s}>> \widetilde{\overline{\mathcal{M}}}_{g,n} \\
@V{\pi_{I_1}\times \pi_{I_2}}VV @VV{\overline{\pi}}V \\
{\mathcal{M}}_{g_1,n_1+1} \times {\mathcal{M}}_{g_2, n_2+1} @>>s> \overline{\mathcal{M}}_{g,n} \\
\end{CD}
\]
thus $\widetilde{S}$ is nothing but
the pre-image $\overline{\pi}^{-1}(S)$ of the image $S$ of $s$, which is indeed a boundary stratum of codimension $1$, and taking the
pre-image along a fiber bundle does not change the codimension.
\end{remark}

\begin{notation}
\label{tubes}
We denote by $\nu: \widetilde{N} \rightarrow \widetilde{S}$ the normal bundle to
$\widetilde{S} = \text{Im}(\tilde{s})$. By the tubular neighbourhood theorem, there exists
a tubular neighbourhood (unique up to isotopy) $Y$ of $\widetilde{S}$
and a homeomorphism $\phi: Y \rightarrow \widetilde{N}$
that commutes with the respective inclusions. Thus we get a map $\nu \circ \phi : Y \rightarrow \widetilde{S}$.
As a consequence of the tubular neighbourhood theorem, the boundary $\partial Y$
is homeomorphic to a subspace of $Y$ by means of a homeomorphism $\psi$
such that $\nu\circ\phi\circ \psi = \nu \circ \phi$. Therefore under the identification $\phi$, $\partial Y$
corresponds to a circle subbundle $\widetilde{\partial N}$ of $\widetilde{N}$.
In the sequel, we suppose to have fixed $Y$ and $\phi$, and we forget about the homeomorphism: we will
write $\widetilde{N}$ to denote both the normal bundle to $\widetilde{S}$ and the associated tubular
neighbourhood $Y$. We therefore write $\nu$ to denote both the normal bundle map, and the map that
we have previously written $\nu \circ \phi$. With $\widetilde{\partial N}$ we will mean both
the circle subbundle of $\widetilde{N}$ and the circular neighbourhood of $\widetilde{S}$ that we have previously
denoted by $\partial Y$.
\end{notation}

\begin{lemma}
\label{normbundle}
Let $S^1$ act on $\widetilde{\mathcal{M}}_{g_1, n_1+1}$ at the $(n_1+1)$-th marked point and on
$\widetilde{\mathcal{M}}_{g_2, n_2+1}$ at the $(n_2+1)$-th marked point respectively, as in Remark \ref{action}.
Let $\nu:\widetilde{\partial N} \rightarrow \widetilde{S}$ be the circle subbundle of $\widetilde{N}$
as in Notation \ref{tubes}, and let
$\widetilde{\mathcal{M}}_{g_1, n_1+1} \times_{S^1} \widetilde{\mathcal{M}}_{g_2, n_2+1}$ be the quotient
of $\widetilde{\mathcal{M}}_{g_1,n_1+1} \times \widetilde{\mathcal{M}}_{g_1,n_2+1}$ by the identification
$$(C'([v_1], \ldots, [v_{n_1}] ,[z v']), C''([v_{n_1+1}, \ldots, [v_n], [v'']))
\sim $$
$$\sim(C'([v_1], \ldots [v_{n_1}], [v'] , C''([v_{n_1+1}, \ldots, [v_n], [zv''])).$$
Then the map 
$$f : \widetilde{\mathcal{M}}_{g_1, n_1+1} \times_{S^1} \widetilde{\mathcal{M}}_{g_2, n_2+1}
\rightarrow \widetilde{S}$$
that forgets the tangent directions and sews the points marked $n_1+1$ and $n_2+1$ respectively,
is a bundle map isomorphic to $\nu$.
\end{lemma}

\begin{proof}
Since by Remark \ref{codim1}, $\widetilde{N} = \pi^*N $ where
$N$ is the normal bundle of $S$ in $\overline{\mathcal{M}}_{g,n}$,
 we see that if $C = \tilde{s}(C', C'') \in \widetilde{S}$ then 
the fiber of $\widetilde{N}$ at $C$
is $T_{x_{n_1+1}}C' \otimes_{\mathbb{C}} T_{x_{n_2+1}}C''$. 

We define a map $ \widetilde{N} \rightarrow \widetilde{\mathcal{M}}_{g_1, n_1+1}
\times_{S^1} \widetilde{\mathcal{M}}_{g_2, n_2+1} $ by writing
$$ (\tilde{s}(C',C''), [v_1], \ldots, [v_n], v'\otimes v'') \mapsto
((C', [v_1], \ldots, [v_{n_1}], [v']), (C'', [v_{n_1+1}], \ldots, [v_n], [v'']))$$ 
This maps is surjective and becomes bijective when restricted to $\widetilde{\partial N}$.
The fact that this gives an isomorphism
between the two bundles is clear by the definition of this map.
\end{proof}

\begin{proposition}
\label{smoothing}
There exists a \emph{sewing-smoothing} map
$$ \sigma : \widetilde{\mathcal{M}}_{g_1, n_1+1} \times \widetilde{\mathcal{M}}_{g_2, n_2+1}
\rightarrow \widetilde{\mathcal{M}}_{g,n}$$
and a circular neighbourhood $\nu: \widetilde{\partial N}\rightarrow \widetilde{S}$  as in
Notation \ref{tubes}, such that
\begin{enumerate}
\item $\text{\emph{Im}}(\sigma) \subseteq \widetilde{\partial N}$,
\item $\nu\circ \sigma=\tilde{s}\circ (\pi^{I_1} \times \pi^{I_2})$.
\end{enumerate}
\end{proposition}

\begin{proof}
The map $\sigma$ is obtained
by composing the quotient map
$$ \widetilde{\mathcal{M}}_{g_1, n_1+1} \times \widetilde{\mathcal{M}}_{g_2, n_2+1} \rightarrow
\widetilde{\mathcal{M}}_{g_1, n_1+1} \times_{S^1} \widetilde{\mathcal{M}}_{g_2, n_2+1}$$
with the isomorphism of Lemma \ref{normbundle}
$$ \widetilde{\mathcal{M}}_{g_1, n_1+1} \times_{S^1} 
\widetilde{\mathcal{M}}_{g_2, n_2+1} \rightarrow \widetilde{\partial N}$$
The verification of properties $1$ and $2$ is immediate.

\end{proof}

To define a cohomological field theory, we must be able to deal with self-sewing maps as well. Thus, let
$J = \{1, \ldots, n \}$ and let
$$ \tilde{q} : \widetilde{\mathcal{M}}^J_{g-1,n+2} \rightarrow \widetilde{\overline{\mathcal{M}}}_{g,n}$$
be the map that sews together the points marked $n+1$ and $n+2$.
We stress again the fact that they are free points.
Let $\widetilde{T}$ be the image of $\tilde{q}$.
Then, as in Remark \ref{codim1}, $\widetilde{T}$ is a boundary stratum
of $\widetilde{\overline{\mathcal{M}}}_{g,n}$ of codimension $1$, and the proof is exactly the same.

\begin{notation}
\label{tubes1}
We call $\nu_q: \widetilde{P} \rightarrow \widetilde{T}$ the normal bundle of $\widetilde{T}$
and the relative tubular neighbourhood given by the tubular neighbourhood theorem. We call
$\widetilde{\partial P}$ be the boundary of $\widetilde{P}$ and the circular subbundle of $\nu_q$.
The justification for these identifications can be found in Notation \ref{tubes}.
\end{notation}

\begin{remark}
 Recall that we are talking about orbifolds: let us consider the vector bundle
  $\widetilde{P}' = T_{x_{n+1}} \otimes T_{x_{n+2}}$ on $\widetilde{\mathcal{M}}_{g-1,n+2}$.
 We have a $\mathbb{Z}_2$-action on $\widetilde{P}'$
 generated by
 $$((C,x_1, \ldots ,x_n, x_{n+1}, x_{n+2}), (u \otimes v)) \mapsto 
 ((C,x_1, \ldots ,x_n, x_{n+2}, x_{n+1}), (v \otimes u)).$$
  Then the normal bundle $\widetilde{P}$ to
 $\widetilde{T}$ is, by definition, the quotient of $\widetilde{P}'$ by this action. Its fiber over a curve 
 $C \in \widetilde{\mathcal{M}}_{g,n}$ is $T_{x_{n+1}}C' \otimes T_{x_{n+2}}C'$ where 
 $C' \in \widetilde{\mathcal{M}}_{g-1,n+2}$ is any of the two normalizations of $C$.
 The tubular neighbourhood theorem still holds for this orbifold version of the normal bundle.
\end{remark}

This is perhaps the best place to state, without proof, the following fundamental theorem.

\begin{theorem}[Harer's stability]
  Let $s_0 : \widetilde{\mathcal{M}}_{g,n} \rightarrow \widetilde{\mathcal{M}}_{g,n+1}$ be the map that sews
  together the curves of$\widetilde{\mathcal{M}}_{g,n}$ 
  and $\widetilde{\mathcal{M}}_{0,3}$ (which is a space consisting
  of just one point) along the points marked $n$ and $3$ respectively. Let
  $p: \widetilde{\mathcal{M}}_{g,n+1} \rightarrow \widetilde{\mathcal{M}}_{g,n}$ be the forgetful map. 
  Then, if $k \le g/3$, the induced map
  in cohomology
  $$s_0^* : H^k(\widetilde{\mathcal{M}}_{g,n+1}) \rightarrow H^k(\widetilde{\mathcal{M}}_{g,n})$$
  is an isomorphism, with inverse $p^*$.
\end{theorem}

\begin{proof}
 See \cite{harer:stability}.
\end{proof}

By this theorem, we call the range of degree less than $g/3$ the \emph{stable range} of the cohomology
of $\overline{\mathcal{M}}_{g,n}$.

The proofs of the following lemma, corollary and proposition are very similar to those of
Lemma \ref{normbundle} and Proposition \ref{smoothing}
respectively, therefore they will be omitted.

\begin{lemma}
\label{autosew}
Let $(\widetilde{\mathcal{M}}_{g-1,n+2})_{S^1}$ be the quotient space of $\widetilde{\mathcal{M}}_{g-1,n+2}$
obtained by
 identifying $(C, [v_1], \ldots, [z\cdot v_{n+1}], [v_{n+2}])$
with $(C, [v_1], \ldots, [v_{n+1}], [z\cdot v_{n+2}])$
for $z \in S^1$ that acts by multiplication on $S(T_{x_i})$. Then the map
$$g: (\widetilde{\mathcal{M}}_{g-1,n+2})_{S^1} \rightarrow \widetilde{T}$$
that forgets the tangent directions and then applies $\tilde{q}$ is isomorphic to the circle subbundle 
$\nu_q: \widetilde{\partial P} \rightarrow \widetilde{T}$.
\end{lemma}

\begin{proposition}
\label{smoothing2}
There exists a sewing-smoothing map
$$ \tau : \widetilde{\mathcal{M}}_{g-1,n+2}
\rightarrow \widetilde{\mathcal{M}}_{g,n}$$
and a circular neighbourhood $\nu_q: \widetilde{\partial P} \rightarrow \widetilde{T}$ as in Notation \ref{tubes1},
such that
\begin{enumerate}
\item $\text{\emph{Im}}(\tau) \subseteq \widetilde{\partial P}$,
\item $\nu_q \circ \tau = \tilde{q} \circ \pi^J$
\end{enumerate}
\end{proposition}

We are now ready to define fixed boundary cohomological field theories.

\begin{definition}
\label{fixedboundary}
Let $A$ be a $\mathbb{C}$-vector space of finite dimension, $\mathbf{1} \in A$ a non-zero vector,
and $\eta = \eta_{\mu \nu}e^{\mu}\otimes e^{\nu}$ 
a nondegenerate symmetric bilinear form on $A$ with inverse co-form $\eta^{\mu \nu} e_{\mu} \otimes e_{\nu}$.
A \emph{fixed boundary cohomological field theory} (fixed boundary CohFT) with base $(A,\eta, \mathbf{1})$
is a set of linear maps

$$\widetilde{\Omega}_{g,n}: A^{\otimes n} \rightarrow H^{\bullet}(\widetilde{\mathcal{M}}_{g,n})$$
for every nonnegative
integer numbers $n$, $g$ such that $2g - 2 + n > 0$, that satisfy the following conditions:

\begin{enumerate}
 \item $\widetilde{\Omega}_{g,n}$ is $S_n$-equivariant for every $g$, $n$ (compare with Remark \ref{equivariance});
 \item $\widetilde{\Omega}_{0,3}(\mathbf{1} \otimes u \otimes v) = 
 \eta(u,v) \in H^0(\widetilde{\mathcal{M}}_{0,3})$, for any $u,v \in A$;
 \item $\tau^{*}\widetilde{\Omega}_{g,n}(v_1 \otimes \ldots \otimes v_n) = 
 \eta^{\mu \nu}\widetilde{\Omega}_{g-1, n+2}(v_1 \otimes \ldots \otimes v_n \otimes e_{\mu} \otimes e_{\nu})$,
 for any $v_1, \ldots, v_n \in A$, where $\tau$ is the map of Proposition \ref{smoothing2};
 \item $\sigma^{*}\widetilde{\Omega}_{g,n}(v_1 \otimes \ldots \otimes v_n) = 
 \eta^{\mu \nu} \widetilde{\Omega}_{g_1, n_1 + 1}(v_1 \otimes \ldots \otimes v_{n_1} \otimes e_{\mu}) \times
 \widetilde{\Omega}_{g_2, n_2 + 1}(v_{n_1 + 1} \otimes \ldots \otimes v_n \otimes e_{\nu})$,
 for any $v_1, \ldots, v_n \in A$, where $\sigma$ is the map of Proposition \ref{smoothing};
 \item $p^{*}\widetilde{\Omega}_{g,n}(v_1 \otimes \ldots \otimes v_n) 
 = \widetilde{\Omega}_{g,n+1}(v_1 \otimes \ldots \otimes v_n
 \otimes \mathbf{1})$,
 for any $v_1, \ldots, v_n \in A$, where 
 $p: \widetilde{\mathcal{M}}_{g,n+1} \rightarrow \widetilde{\mathcal{M}}_{g,n}$ is the map
 that forgets the last marked points, together with the tangent direction attached to it.
\end{enumerate}
The summation over $\mu$ and $\nu$ is always assumed.

\end{definition}

\subsection{Free boundary theories}
Now we consider free boundary theories. This means that we don't assign any tangent
vector to the marked points and we look at homomorphisms
$$\Omega_{g,n}: A^{\otimes n} \rightarrow H^{\bullet}(\mathcal{M}_{g,n})$$
Notice that we are still talking about \emph{smooth} theories, that is, theories in which only the cohomology
ring of the smooth part of the moduli space of curves is involved.
Let
$$s: \mathcal{M}_{g_1,n_1+1} \times \mathcal{M}_{g_2,n_2+1} \rightarrow \overline{\mathcal{M}}_{g,n}$$
be the map that sews together the points marked $n_1+1$ and $n_2+1$
respectively on the two curves, and let $S$ be the
image of $s$.

\begin{notation}
\label{tubeneighfree}
We call $\nu_s: N_s \rightarrow S$ both the normal bundle to $S$ and a fixed tubular neighbourhood of $S$ given
by the tubular neighbourhood theorem. We call $\nu_s: \partial N_s \rightarrow S$ the boundary of the tubular
neighbourhood $N_s$ and the circle subbundle of $\nu_s$ given by the tubular neighbourhood theorem.
Let
$$q: \mathcal{M}_{g-1,n+2} \rightarrow \overline{\mathcal{M}}_{g,n}$$
be the self-sewing map, and let $T$ be its image. We call
$$ \nu_q : N_q \rightarrow T$$
both its normal bundle and its respective tubular neighbourhood given by the tubular neighbourhood theorem.
We call $\nu_q : \partial N_q \rightarrow T$ both the circular neighbourhood boundary of $N_q$
and the restriction of the normal bundle to the circle subbundle given by the tubular neighbourhood theorem.
\end{notation}

\begin{remark}
\label{2covering}
Clearly, $S \simeq \mathcal{M}_{g_1, n_1+1} \times \mathcal{M}_{g_2, n_2+1}$. However,
the self-sewing map $$q: \mathcal{M}_{g-1,n+2} \rightarrow T$$ is a $2$-sheeted covering, where
two marked curves 
$$(C_1, x_1, \ldots, x_{n+2}), \ (C_2, y_1, \ldots, y_{n+2}) \in \mathcal{M}_{g-1,n+2}$$
belong to the same fiber of $q$ if and only if $C_1 = C_2$, $x_i = y_i$ for $1 \le i \le n$ and
$x_{n+1} = y_{n+2}$, $x_{n+2} = y_{n+1}$. Therefore the automorphism
$\rho$ of $\mathcal{M}_{g-1,n+2}$ that switches the two last marked points is the only nontrivial automorphism
of the covering $q$.
By the general theory of covering maps, a cohomology class $\alpha$  on $\mathcal{M}_{g-1,n+2}$ is $\rho$-invariant
if and only if it is the pull-back of a cohomology class on $T$.
Indeed, under the hypothesis of $\rho$-invariance, we have
$$ \alpha = q^*\left(\frac{1}{2}q_*\alpha\right) $$
where $q_*$ is the push-forward map in cohomology, which is well-defined since $q$ is proper.

\end{remark}

Let us now define the free boundary cohomological field theories.

\begin{definition}
\label{freeboundary}
Let $A$ be a $\mathbb{C}$-vector space of finite dimension, $\mathbf{1} \in A$ a non-zero vector,
and $\eta = \eta_{\mu \nu}e^{\mu}\otimes e^{\nu}$ 
a nondegenerate symmetric bilinear form on $A$ with inverse co-form $\eta^{\mu \nu} e_{\mu} \otimes e_{\nu}$.
A \emph{free boundary cohomological field theory} (free boundary CohFT) with base $(A,\eta, \mathbf{1})$
is a set of linear maps

$$\Omega_{g,n}: A^{\otimes n} \rightarrow H^{\bullet}(\mathcal{M}_{g,n})$$
for every nonnegative
integer numbers $n$, $g$ such that $2g - 2 + n > 0$, that satisfy the following conditions:

\begin{enumerate}
 \item $\Omega_{g,n}$ is $S_n$-equivariant for every $g$, $n$;
 \item $\Omega_{0,3}(\mathbf{1} \otimes u \otimes v) = \eta(u,v)$, for every $u,v \in A$;
 \item $\Omega_{g,n}(v_1 \otimes \ldots \otimes v_n)|_{\partial N_q} =
\nu_q^*(\frac{1}{2}q_*\eta^{\mu\nu}\Omega_{g-1, n+2}(v_1 \otimes \ldots \otimes v_n \otimes e_{\mu} \otimes e_{\nu}))$,
 for any $v_1, \ldots, v_n \in A$ (see Remark \ref{2covering} for the reason of the map $\frac{1}{2}q^*$);
 \item $\Omega_{g,n}(v_1 \otimes \ldots \otimes v_n) |_{\partial N_s} 
= \nu_s^{*}\eta^{\mu \nu}\Omega_{g_1,n_1+1}(v_1 \otimes \ldots \otimes v_{n_1} \otimes e_{\mu})
\times \Omega_{g_2,n_2+1}( v_{n_1+1} \otimes \ldots \otimes v_n \otimes e_{\nu})$,
 for any $v_1, \ldots v_n \in A$;
 \item $p^{*}{\Omega}_{g,n}(v_1 \otimes \ldots \otimes v_n) 
 = {\Omega}_{g,n+1}(v_1 \otimes \ldots \otimes v_n
 \otimes \mathbf{1})$,
 for any $v_1, \ldots, v_n \in A$, where $p$ is the map that forgets the last marked point.
\end{enumerate}

\end{definition}

\begin{remark}
\label{nos}
In axiom $4$, we are identifying $S$ with $\mathcal{M}_{g_1, n_1+1} \times \mathcal{M}_{g_2, n_2+1}$
by means of the map $s$, which is an homeomorphism with its image.
\end{remark}

The next proposition allows us to lift free boundary CohFTs to fixed boundary ones along the torus bundle.

\begin{proposition}
\label{coherfree}
 Let $(\Omega_{g,n})_{g,n}$ be a free boundary CohFT , and let 
 $\pi: \widetilde{\mathcal{M}}_{g,n} \rightarrow \mathcal{M}_{g,n}$
 be the map forgetting the tangent directions. Then
 $(\pi^{*}\Omega_{g,n})_{g,n}$ define a fixed boundary CohFT .
\end{proposition}

\begin{proof}
We have to show that the classes $(\pi^{*}\Omega_{g,n})_{g,n}$ satisfy the axioms of a cohomological field theory
with fixed boundaries.
The bilinear form $\eta$ and the unity vector $\mathbf{1}$ stay the same in the two cases and the axiom
involving the forgetful map is trivially satisfied, thus we only have
to deal with the sewing axiom. We have to
show that 
$$\sigma^* \pi^* \Omega_{g,n} =
\eta^{-1}\pi^*\Omega_{g_1,n_1+1}
\times \pi^*\Omega_{g_2,n_2+1}$$
where $\sigma$ is as in Proposition \ref{smoothing} and
$\eta^{-1} = \eta^{\mu\nu}e_{\mu}\otimes e_{\nu}$ is inserted at the entries corresponding to the points
to be sewed.

If $N_s$ is the normal bundle to $S$, then  
$\widetilde{N} = \pi^*N_s$ where $\nu: \widetilde{N} \rightarrow \widetilde{S}$
is the normal bundle to $\widetilde{S}$ (notations as in
Lemma \ref{normbundle}).
Thus we can suppose to have chosen a circular neighbourhoods
such that $\partial N_s = \pi(\widetilde{\partial N})$, therefore $\text{Im}(\pi\circ \sigma) \subseteq
\partial N_s$ by Proposition \ref{smoothing}.

Therefore we have, by the property of the $\Omega_{g,n}$'s of satisfying the sewing axiom
$$\sigma^*\pi^*\Omega_{g,n} = \sigma^*\pi^*\Omega_{g,n}|_{\partial N_s}
= \sigma^*\pi^*\nu_s^*(\eta^{-1}\Omega_{g_1,n_1+1} \times \Omega_{g_2,n_2+1}) $$
but $\nu_s \circ \pi = \pi \circ \nu$, and by Proposition \ref{smoothing},
$\nu \circ \sigma = \tilde{s} \circ (\pi^{I_1} \times \pi^{I_2})$,
and in turn $\pi \circ \tilde{s} = s \circ (\pi_{I_1} \times \pi_{I_2})$,
and trivially $(\pi_{I_1} \times \pi_{I_2}) \circ (\pi^{I_1} \times \pi^{I_2}) =
\pi \times \pi$, thus altogether we have 
$\nu_s \circ \pi \circ \sigma = s \circ (\pi \times \pi)$. Because of Remark \ref{nos},
the map $s^*$ accounts for the identification of $S$ with $\mathcal{M}_{g_1, n_1+1} \times \mathcal{M}_{g_2, n_2+1}$,
therefore can omit it in the sewing axiom and we can rewrite the right-hand side of the equality above as
$$(\pi\times \pi)^*(\eta^{-1}\Omega_{g_1,n_1+1} \times \Omega_{g_2,n_2+1})
= \eta^{-1}(\pi^*\Omega_{g_1,n_1+1} \times \pi^*\Omega_{g_2,n_2+1})$$
that is, the classes $(\pi^{*}\Omega_{g,n})_{g,n}$ satisfy the separating sewing axiom. 

Let now $q: \mathcal{M}_{g-1,n+2} \rightarrow T$ be the self-sewing map and 
$\tau: \widetilde{\mathcal{M}}_{g-1,n+2} \rightarrow \widetilde{\mathcal{M}}_{g,n}$ 
as in Proposition \ref{smoothing2}. We have to show that 
$$\tau^* \pi^* \Omega_{g,n} =
\eta^{-1}\pi^*\Omega_{g-1, n+2}$$
with $\eta^{-1}$ inserted at the entries corresponding to the points to be sewed.
Just like before, we can suppose to have chosen circular neighbourhoods
$\widetilde{\partial P}$ of $\widetilde{T}$ (see Notation \ref{tubes1}) and $\partial N_q$ of $T$ such that
$\pi(\widetilde{\partial P}) \subseteq \partial N_q$, therefore
$\text{Im}(\pi\circ \tau) \subseteq \partial N_q$ by Proposition \ref{smoothing2}.
Therefore we have, by the property of the $\Omega_{g,n}$'s of satisfying the sewing axiom,
$$\tau^* \pi^* \Omega_{g,n} = \tau^* \pi^* \Omega_{g,n}|_{\partial N_q} =
\tau^* \pi^*\nu_q^*\left(\frac{1}{2}q_* \eta^{-1}\Omega_{g-1,n+2}\right)$$ 
By the same reasoning as in the separating case, we have $\nu_q \circ \pi \circ \tau = q \circ \pi$;
moreover, since by hypothesis $\Omega_{g-1,n+2}$ is $S_{n+2}$-invariant, we have by Remark
\ref{2covering} that
$$ q^*\left(\frac{1}{2}q_*\Omega_{g-1,n+2}\right) = \Omega_{g-1,n+2} $$
which finally implies that $(\pi^*\Omega_{g,n})_{g,n}$ satisfy the non-separating sewing axiom as well. 
This completes the proof.
\end{proof}

The following proposition implies, with the previous one, that these axioms agree with those of a nodal CohFT.

\begin{proposition}
\label{cohernodal}
Let $\overline{\Omega}_{g,n}: A^{\otimes n} \rightarrow H^{\bullet}(\overline{\mathcal{M}}_{g,n})$
for $2g - 2 + n > 0$ be homomorphisms defining a nodal CohFT. Then the restrictions of $\overline{\Omega}_{g,n}$ to 
$\mathcal{M}_{g,n}$ define a smooth CohFT with free boundaries.
\end{proposition}

\begin{proof}
Let $\iota : \mathcal{M}_{g,n} \rightarrow \overline{\mathcal{M}}_{g,n}$ be the inclusion, we have to show that
$\iota^* \overline{\Omega}_{g,n}$ satisfy the axioms of a free boundary CohFT. The only nontrivial verification
is the one about the sewing axiom. Let us consider a sewing map 
$s: \mathcal{M}_{g_1, n_1+1} \times \mathcal{M}_{g_2, n_2+1} \rightarrow \overline{\mathcal{M}}_{g,n} $
and let $S$ be its image. Then, since $\partial N_s \subseteq \mathcal{M}_{g,n}$,
we have (again identifying the circular neighbourhood
$\partial N_s$ with the respective spherized bundle)
$$ \iota^*\overline{\Omega}_{g,n}(v_1 \otimes \ldots \otimes v_n) |_{\partial N_s} = 
\overline{\Omega}_{g,n}(v_1 \otimes \ldots \otimes v_n)|_{\partial N_s} = 
\nu_s^* \overline{\Omega}_{g,n}(v_1 \otimes \ldots \otimes v_n)|_S $$
now we have the isomorphism $s: \mathcal{M}_{g_1, n_1+1} \times \mathcal{M}_{g_2,n_2+1} \cong S $
and of course we can write $s = \overline{s} \circ (\iota \times \iota)$ where $\overline{s}$ is the
sewing map on the Deligne-Mumford compactifications (this map is involved in the separating sewing axiom of 
a nodal CohFT). Under the isomorphism $s$ we have
$$\overline{\Omega}_{g,n}(v_1 \otimes \ldots \otimes v_n)|_S = 
\eta^{\mu\nu}\iota^*\overline{\Omega}_{g_1, n_1+1}(v_1 \otimes \ldots \otimes v_{n_1} \otimes e_{\mu}) 
\times \iota^*\overline{\Omega}_{g_2,n_2+1}(v_{n_1+1}\otimes \ldots \otimes v_n \otimes e_{\nu})$$
by the separating sewing axiom that the $\overline{\Omega}_{g,n}$'s satisfy being a nodal CohFT,
and pulling back by $\nu_s$ we have exactly the separating sewing axiom for smooth fixed boundary theories. 

For the non-separating sewing map $q: \mathcal{M}_{g-1,n+2} \rightarrow \mathcal{M}_{g,n}$, with the
notation of the preceding part, we have
$$ \iota^*\overline{\Omega}_{g,n}(v_1 \otimes \ldots \otimes v_n) |_{\partial N_q} = 
\overline{\Omega}_{g,n}(v_1 \otimes \ldots \otimes v_n)|_{\partial N_q} = 
\nu_q^* \overline{\Omega}_{g,n}(v_1 \otimes \ldots \otimes v_n)|_T $$
Obviously we have $q = \overline{q} \circ \iota$ where $\overline{q}$ is the non-separating sewing map
on the Deligne-Mumford compactification (this map is involved in the non-separating sewing axiom of 
a nodal CohFT).
Since $q: \mathcal{M}_{g-1,n+2} \rightarrow T$ is a $2$-sheeted covering, and since every $\overline{\Omega}_{g,n}$
is $S_n$-invariant, we have by Remark \ref{2covering}
$$\overline{\Omega}_{g,n}(v_1 \otimes \ldots \otimes v_n)|_T = 
\frac{1}{2}q_*(q^* \overline{\Omega}_{g,n}(v_1 \otimes \ldots \otimes v_n)|_T)=
\frac{1}{2}q_*(\eta^{\mu\nu}\iota^*\overline{\Omega}_{g-1,n+2}
(v_1, \otimes \ldots \otimes v_n \otimes e_{\mu} \otimes e_{\nu}))$$
by the non-separating sewing axiom that the $\overline{\Omega}_{g,n}$'s satisfy being a nodal CohFT.
Pulling back by $\nu_s$ we have exactly the non-separating sewing axiom for smooth fixed boundary theories,
and the proof is complete.
\end{proof}

\section{Givental's group action}
\label{Givental's group action}
 
 In this section, we introduce a particular group action on the category of all nodal CohFTs that allows one to
 build new CohFTs from old ones. To do so, we must investigate more deeply the structure of the
 boundary strata of $\overline{\mathcal{M}}_{g,n}$, by introducing the concept of \emph{dual graph}.
 
 \subsection{Dual graphs}
 Let $S \subseteq \overline{\mathcal{M}}_{g,n}$ be an open boundary stratum. Let $C \in S$ be a stable curve
 and let $\nu : C' \rightarrow C$ be its normalization.
 We associate the following objects to $C$:
 \begin{enumerate}
  \item the set $V$ of the components of $C'$ and,
  for each component $v \in V$, its genus $\text{g}(v)$;
  \item the set $H$ of the special points of $C'$ (recall that they are the points in $\nu^{-1}(p)$ where
  $p \in C$ is either a marked point or a node);
  \item a function $\text{v}: H \rightarrow V$ that sends every special point to the component of $C'$
  it belongs to;
  \item an involution $\iota: H \rightarrow H$ that sends every special point to its conjugate (that is, 
  $\iota(h) = h'$ if $\{h, h'\} = \nu^{-1}(p)$ where $p \in C$ is a node,
  and $\iota(h) = h$ if $\{ h \} =\nu^{-1}(p)$ where $p \in C$ is a marked point);
  let then $E = \{ \{h, \iota(h)\} \  | \ h \in H  \ \text{and} \ \iota(h) \neq h\}$ 
  and $L = \{ h \in H  \ |\ \iota(h) = h\}$;
 \end{enumerate}

 \begin{definition}
 \label{dualgrafcurve}
  Let $C \in \overline{\mathcal{M}}_{g,n}$ be a stable curve. The \emph{dual graph}
  of $C$ is the graph $\Gamma_C$ with vertex set $V$, edge set $E$ and leg set $L$, such that
  \begin{itemize}
   \item  two vertices
  $v_1$ and $v_2$ are linked by an edge $e  \in E$
  if and only if $\text{v}(e) = \{ v_1, v_2 \}$;
  \item a leg $l \in L$ belongs to the vertex $v$ if and only if $\text{v}(l) = v$.
  \end{itemize}
  
  An \emph{automorphism} of $\Gamma_C$ is a permutation of the sets $V$ and $H$ which leaves invariant
  $\text{g}$, $\text{v}$ and $\iota$. The group of automorphisms of $\Gamma_C$ is denoted by
  $\text{Aut}(\Gamma_C)$.
\end{definition}

\begin{remark}
\label{connected}
 If $C \in \overline{\mathcal{M}}_{g,n}$ is a stable curve with normalization $\nu: C' \rightarrow C$,
 then its graph $\Gamma_C$ is connected because $C$ is connected, that is, we can reach any component of its from
 any other one, by passing through its nodes. Dually, this means that we can reach any vertex in $\Gamma_C$
 from any other one, by passing through its edges.
\end{remark}

\begin{remark}
\label{stability}
Let $C\in \overline{\mathcal{M}}_{g,n}$ be a stable curve and
 for each vertex $v$ of $\Gamma_C$, let $\text{n}(v) = |\text{v}^{-1}(v)|$ be its \emph{valence}.
 Then the stability condition on $C$ says that, for each vertex $v$ of
 $\Gamma_C$, we must have
 $$ 2 \text{g}(v) - 2 + \text{n}(v) > 0 $$
\end{remark}
 
 \begin{remark}
 \label{assocstratum}
  All the curves in an open boundary stratum $S \subseteq \overline{\mathcal{M}}_{g,n}$
  have the same dual graph. Indeed,
  if $C \in S$ is any curve in the stratum with dual graph $\Gamma_C$, let $\text{n}(v)$ be its valence
  as in remark \ref{stability} for each $v \in V$. Then
  $S$ is the image of a map
  $$ \xi_{\Gamma_C} : \prod_{v \in V} \mathcal{M}_{\text{g}(v), \text{n}(v)}
  \rightarrow \overline{\mathcal{M}}_{g,n}$$
  defined in the following way:
  \begin{itemize}
   \item If there are $k$ edges $e_1, \ldots, e_k$ with $e_i = \{ h_i, \iota(h_i) \}$
   between two \emph{distinct} vertices $v_1$ and $v_2$, then we sew together $\mathcal{M}_{g(v_1), n(v_1)}$
   and $\mathcal{M}_{g(v_2), n(v_2)}$ along the special points $h_i$ and $\iota(h_i)$ respectively,
   for $1 \le i \le k$. 
   \item If there are $k$ edges $e_1, \ldots, e_k$ with $e_i = \{ h_i, \iota(h_i) \}$ that link 
   $v$ to itself (this means that $\text{v}(e_i) = \{ v \}$ for every $i$), then we apply a
   non-separating sewing map on $\mathcal{M}_{g(v), n(v)}$ by sewing together the special points 
   $h_i$ and $\iota(h_i)$ for $1 \le i \le k$.
     
  \end{itemize}
It is then evident that all the elements of the image of $\xi_{\Gamma_C}$ give rise to a graph which is the same as 
$\Gamma_C$. Thus this remark justifies the following definition.
 \end{remark}

 \begin{definition}
 Let $S \subseteq \overline{\mathcal{M}}_{g,n}$ be an open boundary stratum. Then its \emph{dual graph}
 $\Gamma_S$ is the dual graph $\Gamma_C$ of any curve $C \in S$.
 
 Let $T \subseteq \overline{\mathcal{M}}_{g,n}$ be a closed boundary stratum. Then its \emph{dual graph}
 $\Gamma_T$ is the dual graph $\Gamma_S$ of the open boundary stratum $S$ such that $T = \overline{S}$.
 Since the correspondence $S \mapsto \overline{S}$ is a bijection, in the sequel we forget the distinction
 between open and closed strata, when we talk about the dual graphs associated to them. 
 \end{definition}
 
 \begin{remark}
 \label{assocclostratum}
 If $S$ is a \emph{closed} boundary stratum, with dual graph $\Gamma_S$, then we can define a map
 $$\overline{\xi}_{\Gamma_S} : \prod_{v \in V} \overline{\mathcal{M}}_{\text{g}(v), \text{n}(v)}
 \rightarrow \overline{\mathcal{M}}_{g,n}$$
 with image $S$, in the same way as in Remark \ref{assocstratum}.
 
 \end{remark}
 
 \begin{remark}
 The genus of a curve $C \in \overline{\mathcal{M}}_{g,n}$ is the sum of the genera of its components, 
 plus the maximal length $k$ of a sequence of self-sewing maps
 $q_i : \overline{\mathcal{M}}_{g_i-1, n_i+2} \rightarrow \overline{\mathcal{M}}_{g_i,n_i}$
 such that 
  \begin{equation}
  \label{chainlenght}
  C = q_k \circ \ldots \circ q_1 (C')
  \end{equation}
  for some $C' \in \overline{\mathcal{M}}_{g_1,n_1}$.
 If $C = q_1(C')$, then $\Gamma_C$ is the graph obtained
 by $\Gamma_{C'}$ by attaching a loop to the vertex corresponding to the components sewn by $q_1$.
 Therefore $h^1(\Gamma_C) = h^1(\Gamma_{C'}) + 1$,
 and similarly if $C = q_k \circ \ldots \circ q_1 (C')$ then $h^1(\Gamma_C) = h^1(\Gamma_{C'}) + k$. Moreover
 there exists $C'$ such that $q(C') = C$ if and only if $\Gamma_C$ is not simply connected (it suffices to 
take as $C'$ the normalization of $C$ at a node corresponding of an edge that borders a loop in $\Gamma_C$),
therefore if $k$ is the maximal length of a sequence like (\ref{chainlenght}), then
$h^1(\Gamma_C) = k$.

Putting all together, we have the formula
 $$g = \sum_{v \in V}\text{g}(v) + h^1(\Gamma_C)$$
 the first term accounts for the sum of the genera of the components of $C$, and the second one for the
 non-separating nodes.
 
\end{remark}

\subsection{R-matrix action}

This section follows closely \cite{pandpixtzvon:spin}, where all the missing proofs can be found, along with
a much deeper treatment of the subject.

Let $A$ be a finite-dimensional vector space over $\mathbb{C}$ and $\eta$ a symmetric
non-degenerate bilinear form on $A$.
For a variable $z$, let
$$ R(z) = \text{Id} + R_1 z + R_2 z^2 + \ldots \ \in \text{End}(A)[[z]]$$
be an $\text{End}(A)[[z]]$-valued power series.
Let $R(z)^*$ be the adjoint of $R(z)$ with respect to $\eta$,
that is, $R(z)^* = \eta^{-1}R(z)^t\eta$ where 
$$R(z)^t=\text{Id} + R_1^t z + R_2^t z^2 + \ldots \ \in \text{End}(A)[[z]]$$ 
is the transpose of $R(z)$. 
We say that $R(z)$ satisfies the
\emph{symplectic condition} if
$$ R(z)R(-z)^* = \text{Id}.$$
If $R(z)$ satisfies the symplectic condition, then the quantity
$$\displaystyle\frac{\eta^{-1} - R(z)\eta^{-1}R(w)^t}{z+w}$$
is a \emph{power series} in the variables $z$ and $w$.

Note that $R^{-1}(z) = \frac{1}{R(z)} = R(-z)^*$ is still a power series,
and it satisfies the symplectic condition as well.
Therefore the quantity
$$\displaystyle\frac{\eta^{-1} - R^{-1}(z)\eta^{-1}R^{-1}(w)^t}{z+w} $$ 
is a power series.

If $v \in A$ is a vector, we denote by $R(z)v$ the $A$-valued power series
$$ R(z)v = v + R_1(v)z + R_2(v)z^2 + \ldots \in A[[z]].$$
If $\phi: A \rightarrow B$ is a homomorphism of vector spaces and $v \in A$, we denote by $\phi(R(z)v)$
the $B$-valued power series
$$\phi(R(z)v) = \phi(v) + \phi(R_1(v))z + \phi(R_2(v))z^2 + \ldots \in B[[z]].$$

\begin{definition}
\label{contribution}
Let $(\overline{\Omega}_{g,n})_{g,n}$ be a nodal cohomological field theory with base $(A, \mathbf{1}, \eta)$, and
let $T(z) = z\cdot[\mathbf{1} - R(z)^{-1}\mathbf{1}] \in z^2 A[[z]]$.
Let  $S \subseteq \overline{\mathcal{M}}_{g,n}$ be a closed boundary stratum and let
$\Gamma_S$ be its dual graph, with vertex set $V$. We define
the \emph{contribution} of $S$
$$\text{Cont}_{\Gamma_S} :  A^{\otimes n} \rightarrow H^{\bullet}(\overline{\mathcal{M}}_{g,n})$$
in the following way:
\begin{enumerate}
\item At every vertex $v$ of $\Gamma_S$, we place 
$$\sum_{m \ge 0} \frac{1}{m!}(p_m)_*\overline{\Omega}_{\text{g}(v), \text{n}(v)+m}
(\cdot \otimes T(\psi_{\text{n}(v)+1}) \otimes \ldots \otimes T(\psi_{\text{n}(v)+m}))$$
where $p_m: \overline{\mathcal{M}}_{g,n+m} \rightarrow \overline{\mathcal{M}}_{g,n}$ is the map
that forgets the last $m$ marked points (in the sum, the term corresponding to $m=0$ is simply 
$\overline{\Omega}_{\text{g}(v), \text{n}(v)}$).
\item At every leg $l$ of $\Gamma_S$,
we place $R^{-1}(\psi_l)$ (recall that a leg corresponds to
a marked point, that has a respective $\psi$-class);
\item At every edge $e$ of $\Gamma_S$, we place
$$\displaystyle\frac{\eta^{-1} - R^{-1}(\psi_e')\eta^{-1}R^{-1}(\psi_e'')^t}{\psi_e'+\psi_e''} $$ 
(recall that an edge corresponds to a node, that has two $\psi$-classes attached to the two special points
of the normalization at that node).
\item In this way we have a map $A^{\otimes n} \rightarrow 
\prod_{v \in V}\overline{\mathcal{M}}_{\text{g}(v),\text{n}(v)}$.
We push forward this map along the map $\overline{\xi}_{\Gamma_S}$ of Remark
\ref{assocclostratum}, which is proper, to obtain $\text{Cont}_{\Gamma_S}$. 
\end{enumerate}

\end{definition}

 \begin{definition}
 \label{matrixaction}
 
 Let $\overline{\Omega}_{g,n}$ be a nodal CohFT with base $(A, \mathbf{1}, \eta)$. 
 We define $(R\overline{\Omega})_{g,n} : A^{\otimes n} \rightarrow H^{\bullet}(\overline{\mathcal{M}}_{g,n})$
 with the following formula
 $$(R\overline{\Omega})_{g,n} = 
 \sum_{S\subseteq\overline{\mathcal{M}}_{g,n}}\frac{1}{|\text{Aut}(\Gamma_S)|}\text{Cont}_{\Gamma_S}$$
 where $S$ runs among the closed boundary strata of $\overline{\mathcal{M}}_{g,n}$.
 We say that $R\overline{\Omega}$ is obtained from $\overline{\Omega}$ by an \emph{$R$-matrix action}.
 \end{definition}

\begin{example}
 Let us consider a graph $\Gamma$ with a vertex $v$ of genus $g$, and a leg $l$ attached to it. It
 corresponds to the open stratum $\mathcal{M}_{g,1}$ or to the closed stratum $S = \overline{\mathcal{M}}_{g,1}$.
 Then, for every $a \in A$, we have
 $$\text{Cont}_{\Gamma}(a) = \sum_{m\ge0}\frac{1}{m!}(p_m)_*
 \overline{\Omega}_{g,1+m}(R^{-1}(\psi_l)a \otimes T(\psi_{2}) \otimes \ldots \otimes T(\psi_{1+m})).$$
\end{example}

\begin{example}
 Let us consider a graph $\Gamma$ with a vertex $v$ of genus $g$, no legs, and an edge $e$ that joins $v$ to itself.
 The valence of the only vertex of $\Gamma$ is $2$.
 Then with the convention $A^{\otimes 0} = \mathbb{C}$, we have naturally
 $\text{Cont}_{\Gamma} \in H^{\bullet}(\overline{\mathcal{M}}_g)$ and it is precisely
 $$\text{Cont}_{\Gamma} = \sum_{m\ge 0}\frac{1}{m!}(p_m)_*\eta^{\mu\nu}\overline{\Omega}_{g,2+m}
 \left(\frac{e_\mu \otimes e_\nu - R^{-1}(\psi_e')e_\mu \otimes R^{-1}(\psi_e'')^te_\nu}{\psi_e'+\psi_e''} 
 \otimes T(\psi_{3}) \otimes \ldots \otimes T(\psi_{2+m})\right).$$
\end{example}

\begin{example}
 Let $\Gamma$ be a graph with two vertices $v_1$ and $v_2$ of genera $g_1$ and $g_2$ respectively,
 no legs, and an edge $e$ that
 joins $v_1$ to $v_2$. In this case the valence of each vertex is $1$, and again
 $\text{Cont}_{\Gamma} \in H^{\bullet}(\overline{\mathcal{M}}_g)$. If we set
 $$\text{Cont}_m'(e_\mu) = \frac{1}{m!}(p_m)_*\overline{\Omega}_{g_1, 1+m}
((e_\mu-R^{-1}(\psi_e')e_\mu) \otimes T(\psi_{2}) \otimes \ldots \otimes T(\psi_{1+m}))$$
and $\text{Cont}_m''$ is the same but with $g_2$ and
$R^{-1}(\psi_e'')^t$ instead of $g_1$ and $R^{-1}(\psi_e'$), we have
$$\text{Cont}_{\Gamma} = \frac{\eta^{\mu\nu}}{\psi_e' + \psi_e''}
\sum_{m\ge0}\text{Cont}_m'(e_\mu) \times \sum_{m\ge0}\text{Cont}_m''(e_\nu).$$
\end{example}

 \begin{example}
 Let us consider the case in which $R(z) \equiv \text{Id}$.
 Then the power series of rule $2$ of Definition \ref{contribution}
 is $\eta^{-1} - R^{-1}(z)\eta^{-1}R^{-1}(w)^t = 0$, and $T(z) = 0$.
 Therefore all the contributions are zero, except the one corresponding to the stratum
 $S \subseteq \overline{\mathcal{M}}_{g,n}$ whose graph $\Gamma_S$ 
 has no edges. By Remark \ref{connected} it must have only one vertex, with $n$ legs attached to it, 
 and it must clearly correspond to the \emph{smooth stratum}
 $S = \mathcal{M}_{g,n} \subseteq \overline{\mathcal{M}}_{g,n}$, or to the closed stratum 
 $S = \overline{\mathcal{M}}_{g,n}$.
 In this case it is easily seen that the contribution is simply
 $$\text{Cont}_{\Gamma_S} = \overline{\Omega}_{g,n} .$$
 Therefore we have $\text{Id}\overline{\Omega} = \overline{\Omega}$.
 \end{example}

 \begin{proposition}
 \label{staysfieldtheory}
 If $\overline{\Omega}$ is a nodal CohFT, then $R\overline{\Omega}$ is a nodal CohFT. The $R$-matrix action
 on CohFTs is a left group action.
 \end{proposition}

 \begin{remark}
 \label{alsoworkswithsmooth}
  We could also have defined, in Definition \ref{contribution}, contributions associated to a \emph{free boundary}
  cohomological field theory $\Omega_{g,n}$, and a matrix action $R\Omega$ in Definition \ref{matrixaction} as well,
  using the same formulas as in the nodal case. It turns out that Proposition \ref{staysfieldtheory}
  stays valid, therefore $R\Omega$ will be a (nodal!) cohomological field theory. 
 \end{remark}

 We will now state one lemma that will be needed in the sequel.
 
 \begin{lemma}
 \label{tremendousexponential}
  Let $A$ be a $\mathbb{C}$-algebra and let
  $$M(z) = z(1 - \text{\emph{exp}}(-a_1z - a_2z^2 - \ldots)) \in A[[z]].$$
  Fix $g, n$ integers such that $2g - 2 + n > 0$, and for each $m \ge 0$,
  let $p_m : \mathcal{M}_{g,n+m} \rightarrow \mathcal{M}_{g,n}$ be the map that forgets
  the last $m$ marked points.
  Then in $\mathcal{M}_{g,n}$ we have the equality
  $$\text{\emph{exp}}(a_1 \kappa_1 + a_2 \kappa_2 + \ldots ) = 
  \sum_{m=0}^{\infty}\frac{1}{m!}(p_m)_*(M(\psi_{n+1})\cdots M(\psi_{n+m}))$$
 \end{lemma}

 \begin{proof}
  We need a technical fact about $\kappa$-classes. For each permutation $\sigma \in S_m$,
  let $\sigma = \rho_1 \cdots \rho_l$ be its decomposition in cycles, where also $1$-cycles have been written. For
  each cycle $\rho_i$, define $k_{\rho_i} = \sum_{j \in \rho_i} k_j$. If
  $\kappa_{k_1, \ldots, k_m} = (p_m)_*(\psi_{n+1}^{k_1+1} \cdots \psi_{n+m}^{k_m+1})$,
  then 
  \begin{equation}
  \label{multikappa}
  \kappa_{k_1, \ldots, k_m} = \sum_{\sigma \in S_m} \prod_{i=1}^l \kappa_{k_{\rho_i}}.
  \end{equation}
  For example, $\kappa_{1,1} = \kappa_1^2 + \kappa_2$ and 
  $\kappa_{1,2,3} = \kappa_1\kappa_2\kappa_3 + \kappa_1\kappa_5 + \kappa_2\kappa_4 + \kappa_3^2 + 2\kappa_6$.
  
  Let $B(z) = \frac{M(z)}{z} \in A[[z]]$.
  For each set of $m$ strictly positive integers $\{ k_1, \ldots, k_m \}$, let
  $$b_{k_1, \ldots , k_m} = a_{k_1}\cdots a_{k_m}\kappa_{k_1,\ldots, k_m}z^{k_1+\ldots + k_m}.$$
  If $\sigma \in S_m$ is a permutation whose decomposition in cycles is $\sigma = \rho_1 \cdots \rho_l$, let
  $$b_{k_1, \ldots , k_m, \sigma} = (a_{k_1}\cdots a_{k_m} \prod_{i=1}^l \kappa_{k_{\rho_i}}) z^{k_1 + \ldots + k_m}$$
  then by Equation \ref{multikappa} we have 
  $b_{k_1, \ldots , k_m} = \sum_{\sigma \in S_m}b_{k_1, \ldots , k_m, \sigma}$. Let 
  $$N_m = \sum_{k_1, \ldots, k_m \ge 1} b_{k_1,\ldots,k_m} $$
  and let $F = \sum_{m \ge 0}\frac{N_m}{m!}$. We see that $F$ is the generating function, evaluated at $1$,
  of some combinatorial objects, namely the sets of $m$ numbers $\{k_1, \ldots , k_m\}$,
  with the $i$-th number weighted by $a_{k_i}z^{k_i}$, along with a weighting $\prod_{i=1}^l \kappa_{k_{\rho_i}}$
  where $\rho_1 \cdots \rho_l \in S_m$ is the decomposition in cycles of a permutation in $S_m$.
  By the theory of generating functions, $F = e^G$ where $G$ is the generating function, evaluated at $1$, of
  the sets of $n$ numbers $\{k_1, \ldots, k_n\}$, with the $i$-th number weighted by $a_{k_i}z^{k_i}$, along
  with a weighting $\kappa_{k_{\rho}}$ where $\rho$ is an $n$-cycle in $S_n$.
  Let $\phi: A[[z]] \rightarrow A[[z,\kappa_i]]_{i \ge 1}$ be the $A$-linear map
  that sends every $z^i$ to $z^i \kappa_i$. Since $k_{\rho} = k_1 + \ldots + k_n$
  does not depend on $\rho$, and since there are exactly $(n-1)!$ $n$-cycles in $S_n$,
  we see that 
  $$G = \phi\left(\sum_{n \ge 0}\frac{(n-1)!}{n!}\sum_{k_1, \ldots ,k_n \ge 1}
  a_{k_1}\cdots a_{k_n}z^{k_1+\ldots + k_n}\right)
  = $$ $$=\phi\left(\sum_{n \ge 0} \frac{B(z)^n}{n}\right) = -\phi(\text{log}(1-B(z))) =
  a_1 \kappa_1 z + a_2  \kappa_2 z^2 + \ldots $$
 Setting $z = 1$ in the equality $F = e^G$ we get precisely the formula we need to prove.
 \end{proof}

 \section{Frobenius algebras}

\begin{definition}
Let $(A, \cdot, \mathbf{1})$ be a finite-dimensional associative algebra with unit element $\mathbf{1}$ 
over the field $\mathbb{C}$ (this will be the only case of real interest in our discussion). This algebra
is a \emph{Frobenius algebra} if it is equipped with a non-degenerete bilinear form
$\eta: A \otimes A \rightarrow \mathbb{C}$ such that
$\eta(a\cdot b, c) = \eta(a, b \cdot c) $ for every $a,b,c \in A$. If $\eta$ is symmetric, then
$A$ is called a \emph{symmetric Frobenius algebra}.
\end{definition}

\begin{remark}
\label{i}
We have an isomorphism $i: A \xrightarrow{\sim} A^*$ between $A$ and the dual space $A^*$,
given by the nondegenerate form $\eta$, namely
$i(v) = \eta(v, \cdot)$. We can define an algebra structure on $A^*$ by imposing $i$ to be an 
isomorphism of algebras. Also, on the dual space $A^*$ we have $(\eta^t)^{-1}$
the inverse transpose form, which is still non-degenerate and defines a structure of Frobenius algebra
on the dual space $A^*$ (the verification is immediate).
\end{remark}

\begin{definition}
\label{Frobeniustrace}
By letting $i: A \xrightarrow{\sim} A^*$ as in the previous remark, the functional $\theta = i(\mathbf{1})$
is called the \emph{Frobenius trace} of the algebra. In other words, 
the Frobenius trace is defined as $\theta(a) = \eta(a, \mathbf{1})$ for every $a \in A$.
\end{definition}

\begin{definition}
\label{Eulerclass}
Let $(A, \cdot, \mathbf{1}, \eta)$ be a Frobenius algebra. The vector
$$\alpha = \eta^{\mu\nu}e_{\mu}\cdot e_{\nu}$$
is called the \emph{Euler class} of $A$.
\end{definition}

\begin{remark}
\label{welldef}
If $Ae_{\mu} = f_{\mu}$ for some invertible matrix $A= (A_{\rho}^{\sigma})_{\rho,\sigma}$,
then $f_{\mu} = A_{\mu}^{\rho}e_{\rho}$
and 
$$\eta^{\mu\nu}e_{\mu}\cdot e_{\nu} = \eta^{\mu\nu}(A^{-1}f_{\mu})\cdot(A^{-1}f_{\nu}) = 
\eta^{\mu\nu} A_{\mu}^{\rho}A_{\nu}^{\sigma}(A^{-1}e_{\rho})\cdot(A^{-1}e_{\sigma}) =
(A^t\eta A)^{\rho\sigma}f_\rho \cdot f_\sigma$$
Where $A^t$ is the transpose matrix of $A$. 
Since the matrix $\eta$, in the basis $(f_\rho)_\rho$, turns into the matrix $A^t\eta A$, we see that
the Euler class does not depend on the choice of a basis of $A$, it is therefore determined by the
Frobenius algebra structure.
\end{remark}

\begin{remark}
\label{nonzero}
The Euler class $\alpha$
does not lie in the kernel of the Frobenius trace of $A$; in particular, it is never zero.
Indeed, let $\theta$ be the Frobenius trace of $A$. Then we have
$\theta(\alpha) = \eta^{\mu\nu}\eta(e_{\mu},e_{\nu}) = \eta^{\mu\nu}\eta_{\mu\nu} = \text{dim}A$,
which is not zero.
\end{remark}

\begin{definition}
Let $(A, \cdot, \mathbf{1}, \eta)$ be a Frobenius algebra of complex dimension $k$.
$A$ is said to be \emph{semisimple} if
there exist an orthonormal (with respect to $\eta$) basis $(e_{\mu})_{1 \le \mu \le k}$ and non-zero
complex numbers $\theta_{\mu}$, such that for every $i$ and $j$, we have
\begin{equation}
\label{semisimple}
e_{\mu} \cdot e_\nu = \delta_{\mu\nu}\theta_{\mu}^{-1} e_\mu
\end{equation}

A basis that satisfies the conditions (\ref{semisimple}) is called a \emph{semisimple basis};
the moduli $|\theta_{\mu}|$ are called the \emph{weights} of the algebra.
\end{definition}

 It is clear that in this case
$\theta_{\mu} = \eta(e_{\mu}, \mathbf{1})$ and also $\theta_{\mu} = \theta(e_{\mu})$ in the notation of Definition
\ref{Frobeniustrace}.
Therefore we have $\mathbf{1} = \sum_{\mu} \theta_{\mu} e_{\mu}$ for a semisimple basis $(e_{\mu})_{\mu}$.

\begin{remark}
\label{weights}
The unordered set of weights is well-defined. Indeed, if $f_\mu$ is another semisimple basis,
then there exists an $\eta$-orthogonal matrix
$P$ (i.e. $P^t \eta P = \eta$) such that $f_\mu = Pe_\mu$.
Let us expand the product for general $\mu$ and $\nu$
$$Pe_{\mu} \cdot Pe_\nu = P_\mu^\rho P_\nu^\sigma e_\rho \cdot e_\sigma 
= P_\mu^\rho P_\mu^\rho \theta_\rho^{-1}e_\rho$$
For  $\mu \neq \nu$, we must have $P^\rho_\mu P^\sigma_\nu = 0$ 
for every $\rho$, therefore the matrix $P$ has at most one (and then
exactly one since it must be non-singular)
non-zero entry for each row. Thus, eventually conjugating by a permutation matrix (which
is orthogonal), we can suppose that $P$ is
diagonal. Since $P$ is $\eta$-orthogonal, the entries in the diagonal must be $1$ or $-1$,
therefore the weights remain unchanged since the $\theta_{\mu}$ can only change sign.
In fact we have shown that the semisimple basis is uniquely
determined, up to permutation and to sign switch of the vectors.
\end{remark}

\begin{proposition}
\label{invertible}
Let $(A, \cdot, \mathbf{1}, \eta)$ be a Frobenius algebra. Then $A$ is semisimple if and only if its
Euler class $\alpha$ is invertible.
\end{proposition} 

\begin{proof}
We will prove just the easy part: if the Frobenius algebra is semisimple, then $\alpha$ is invertible.
The other part is proved in \cite{abrams:frobenius2} and will not be needed in the sequel.

Let $A$ be semisimple and let $(e_{\mu})_{\mu}$ be a semisimple basis. Then by Definition \ref{Eulerclass}
and Remark \ref{welldef} we can write $\alpha = \sum_{\mu} \theta_{\mu}^{-1}e_{\mu}$
in this basis. It is then easy to show that 
$\alpha^{-1} = \sum_{\mu} \theta_{\mu}^3e_{\mu}$ is an element such that $\alpha \cdot \alpha^{-1} = \mathbf{1}$,
that is, $\alpha$ is invertible.
\end{proof}

\begin{definition} \label{Frob}
Let $(\Omega_{g,n})_{g,n}$ be a CohFT of any type with base $(A, \eta, \mathbf{1})$ and let
$i : A \rightarrow A^*$ be the isomorphism of Remark \ref{i}. We define a product $\cdot$
on $A$ by writing, for every $v_1, v_2 \in A$,
$$ v_1 \cdot v_2 = i^{-1}\Omega_{0,3}(v_1 \otimes v_2 \otimes \cdot)$$
or in other words, $\eta(v_1 \cdot v_2, v_3) = \Omega_{0,3}(v_1 \otimes v_2 \otimes v_3)$.
\end{definition}

\begin{remark}
\label{noneedtospecify}
We notice that $\Omega_{0,3}$ and $\overline{\Omega}_{0,3}$ both have values in the cohomology ring of
$\mathcal{M}_{0,3} = \overline{\mathcal{M}}_{0,3}$, which is just $\mathbb{C}$, and by Proposition
\ref{cohernodal} they must be the same homomorphism. The same is true for $\widetilde{\Omega}_{0,3}$,
by Proposition \ref{coherfree}; thus we can avoid specifying the type of the CohFT, because every type 
leads to the same algebra structure.
\end{remark}

\begin{proposition}
\label{thereisfrobstructure}
Let $(\Omega_{g,n})_{g,n}$ be a CohFT of any type with base $(A, \eta, \mathbf{1})$. Then 
there exists a unique Frobenius algebra structure on $A$ such that, for every $v_i \in A$, we have
$$ \Omega_{0,3}(v_1 \otimes v_2 \otimes v_3) = \eta(v_1 \cdot v_2, v_3) $$
This Frobenius algebra is symmetric and commutative, and has $\mathbf{1}$ as neutral element.
\end{proposition}

We will need the following lemma.

\begin{lemma}
\label{multiplication}
Let
$$s: \overline{\mathcal{M}}_{g,n} \times \overline{\mathcal{M}}_{0,3} \rightarrow \overline{\mathcal{M}}_{g,n+1}$$
be the map that sews together the points marked $n$ and $3$ respectively.
Then we have
$$s^* \overline{\Omega}_{g,n+1}(v_1 \otimes v_2 \otimes \ldots \otimes v_{n+1}) =
\overline{\Omega}_{g,n}(v_1 \otimes \ldots \otimes v_{n-1} \otimes (v_n \cdot v_{n+1}))$$

\end{lemma}

\begin{proof}
By the sewing axiom of Definition \ref{nodaltheories}, we have
$$s^*\overline{\Omega}_{g,n+1}(v_1 \otimes v_2 \otimes \ldots \otimes v_{n+1}) =
\eta^{\mu \nu}\overline{\Omega}_{g,n}(v_1 \otimes \ldots \otimes v_{n-1} \otimes e_{\nu})
\times \overline{\Omega}_{0,3}(v_n \otimes v_{n+1} \otimes e_{\mu}) $$
and by Definition \ref{Frob}, the last quantity is
$$\eta^{\mu\nu}\eta(v_n \cdot v_{n+1}, e_{\mu}) 
\overline{\Omega}_{g,n}(v_1 \otimes \ldots \otimes v_{n-1} \otimes e_{\nu}) = 
\overline{\Omega}_{g,n}( v_1 \otimes \ldots \otimes v_{n-1} 
\otimes \eta^{\mu\nu}\eta(v_n \cdot v_{n+1}, e_{\mu})e_{\nu})$$
since $\eta^{\mu\nu}\eta(v_n \cdot v_{n+1}, e_{\mu}) e_{\nu} = v_n \cdot v_{n+1}$, we get the required formula.
\end{proof}

\begin{remark}
\label{smooththeoriesaswell}
The same lemma holds for smooth theories as well, provided we use the right sewing axioms.
This means that if
$\sigma:  \widetilde{\mathcal{M}}_{g,n}\times \widetilde{\mathcal{M}}_{0,3} 
\rightarrow \widetilde{\mathcal{M}}_{g,n+1}$ is the sewing-smoothing map of Proposition \ref{smoothing},
then we have
$$ \sigma^*\widetilde{\Omega}_{g,n+1}(v_1 \otimes v_2 \otimes \ldots \otimes v_{n+1}) =
\widetilde{\Omega}_{g,n}(v_1 \otimes \ldots \otimes v_{n-1} \otimes (v_n \cdot v_{n+1}))$$
Analogously, if $\nu_s : \partial N_s \rightarrow S$ are as in Remark \ref{tubeneighfree}, then we have
$$ \Omega_{g,n+1}(v_1 \otimes v_2 \otimes \ldots \otimes v_{n+1})|_{\partial N_s} =
\nu_s^* \Omega_{g,n}(v_1  \otimes \ldots \otimes v_{n-1}\otimes (v_n \cdot v_{n+1}))$$
The proof is literally the same as in the previous lemma, since the right-hand side is the same in the
three sewing axioms.
\end{remark}

\begin{proof}[Proof of Proposition \ref{thereisfrobstructure}]
First of all,
we see that for every $v \in A$, $v \cdot \mathbf{1} = v$. Indeed, we have
$$ v \cdot \mathbf{1} = i^{-1}\Omega_{0,3}(v \otimes \mathbf{1} \otimes \cdot) = 
i^{-1}\Omega_{0,3}(\mathbf{1} \otimes v \otimes \cdot) = i^{-1}\eta(v, \cdot) = v$$
where we have used the $S_3$-equivariance and the second axiom of CohFTs.

Now associativity, which is equivalent to 
$$ \Omega_{0,3}((v_1 \cdot v_2) \otimes v_3 \otimes v_4) = \Omega_{0,3}(v_1 \otimes (v_2 \cdot v_3) \otimes v_4)$$
for every $v_i \in A$.
By Lemma \ref{multiplication} and by $S_3$-equivariance, these two quantities are
$s^*\Omega_{0,4}(v_3 \otimes v_4 \otimes v_1 \otimes v_2)$ and 
$s^*\Omega_{0,4}(v_1 \otimes v_4 \otimes v_2 \otimes v_3)$ respectively,
where
$$s: \overline{\mathcal{M}}_{0,3} \times \overline{\mathcal{M}}_{0,3} \rightarrow \overline{\mathcal{M}}_{0,4}$$
is the sewing map as in Lemma \ref{multiplication}.
If $\rho = (1, 2, 3) \in S_4$ then by $S_4$-equivariance we have
$\rho^*\Omega_{0,4}(v_1 \otimes v_2 \otimes v_3 \otimes v_4) =
\Omega_{0,4}(v_2 \otimes v_3 \otimes v_1 \otimes v_4)$,
but $\rho^*$ is the identity in degree $0$, and
$s^*$ kills every positive-degree cohomology class, since $\overline{\mathcal{M}}_{0,3}$ has dimension $0$. Therefore we have
$$s^*\Omega_{0,4}(v_3 \otimes v_4 \otimes v_1 \otimes v_2) =
s^*\Omega_{0,4}(v_1 \otimes v_4 \otimes v_2 \otimes v_3) $$
and by Lemma \ref{multiplication} we have associativity.

Finally, we have $\eta(v_1 \cdot v_2, v_3) = \eta(v_1, v_2 \cdot v_3)$ for every $v_i \in A$ because
this equality means
$$\Omega_{0,3}((v_1 \cdot v_2) \otimes v_3 \otimes \mathbf{1}) =
\Omega_{0,3}(v_1 \otimes (v_2 \cdot v_3) \otimes \mathbf{1})$$
which is true thanks to associativity.
\end{proof}

 \begin{lemma}
  Let
 $(\Omega_{g,n})_{g,n}$ be a CohFT of any type with base $(A, \eta, \mathbf{1})$, let
 $\alpha$ be its Euler class, and let 
 $q:\overline{\mathcal{M}}_{g-1,n+2} \rightarrow \overline{\mathcal{M}}_{g,n}$ be the self-sewing map.
 Then
 $$ q^* \Omega_{1,1}(v) = \eta(v,\alpha)$$
 In particular, the degree zero part of $\Omega_{1,1}(v)$ is $\eta(v, \alpha)$.
\end{lemma}

 \begin{proof}
  By the self-sewing axiom, we have 
  $$q^* \Omega_{1,1}(v) = \eta^{\mu\nu} \Omega_{0,3}(v \otimes e_{\mu} \otimes e_{\nu}) 
  =\eta^{\mu\nu}\eta(v, e_{\mu}\cdot e_{\nu})=\eta(v,\alpha)$$
  which is the statement.
 \end{proof}

\begin{lemma}
\label{degzero}
Let $(\Omega_{g,n})_{g,n}$ be a semisimple CohFT (of any type) with Euler class $\alpha$. Then the degree-zero
part of $\Omega_{1,2}(v \otimes w)$ is
$$\Omega_{1,2}(v \otimes w)_0  = \eta(v\cdot w, \alpha) $$

\end{lemma}

\begin{proof}
We do the proof for a fixed boundary theory $(\widetilde{\Omega}_{g,n})_{g,n}$, but the proof is the same for
other theories if we take care of using the right sewing axioms.
We consider  the sewing-smoothing map
  $$\widetilde{\mathcal{M}}_{0,3} \times \widetilde{\mathcal{M}}_{1,1} 
  \xrightarrow{\sigma} \widetilde{\mathcal{M}}_{1,2} $$
 
  By Remark \ref{smooththeoriesaswell} we get 
  $$\sigma^*\widetilde{\Omega}_{1,2}(v \otimes w) = 
  \widetilde{\Omega}_{1,1}(v\cdot w )$$
  By the previous lemma, the degree zero part of the last quantity is $\eta(v\cdot w, \alpha)$
  and the formula is proved.
\end{proof}
 
 \subsection{Sewing smooth surfaces}

 The next three propositions show
 the effect of sewing a fixed smooth $g''$-genus surface $\Sigma$ to the curves in 
 $\overline{\mathcal{M}}_{g,n}$. Using fixed boundaries, free boundaries, or nodal theories changes
 the way this result is stated, even though the proofs are just the same in the three cases.
 For the sake of clarity and completeness, we state three different propositions for the three types
 of theories we have discussed.

 \begin{proposition}
 \label{sewfixedtorus}
  Let $g = g' + g''$ and $\Sigma \in \widetilde{\mathcal{M}}_{g'',2}$
  be a smooth genus $g''$ surface with two \emph{framed} points, and let
  $$ \sigma_{{\Sigma}} :\widetilde{\mathcal{M}}_{g',n} \times \Sigma \rightarrow \widetilde{\mathcal{M}}_{g,n}$$
  be the restriction of the sewing-smoothing map $\sigma$ of Proposition \ref{smoothing}
  that sews together the points marked $n$ and $2$ respectively.
  Then for any $v_1, \ldots, v_n \in A$ we have
  $$ {\sigma_{{\Sigma}}}^*\widetilde{\Omega}_{g,n}(v_1 \otimes \ldots \otimes v_n)
  = \widetilde{\Omega}_{g',n}( v_1 \otimes \ldots \otimes v_{n-1} \otimes (\alpha^{g''} \cdot v_n))$$
 \end{proposition}

 \begin{proof}
  Let us first consider the case $g'' = 1$, so $\Sigma$ is a torus with two fixed marked points.
  By the sewing axiom, we have
  $$\sigma^*\widetilde{\Omega}_{g,n}(v_1 \otimes \ldots \otimes v_n) = 
  \eta^{\mu\nu}\widetilde{\Omega}_{g',n}(v_1 \otimes \ldots \otimes v_{n-1} \otimes e_{\mu}) 
  \times \widetilde{\Omega}_{1,2}(v_n \otimes e_{\nu})$$
  Now we must take the restriction of this class (which is in 
  $H^{\bullet}( \widetilde{\mathcal{M}}_{g',n}\times \widetilde{\mathcal{M}}_{1,1} )$) to the closed subset
  $\widetilde{\mathcal{M}}_{g',n} \times \Sigma$, and the result of this operation is just taking the degree-zero
  part in each right factor and summing up; thus the class we are looking for is
  $${\sigma_{{\Sigma}}}^*\widetilde{\Omega}_{g,n}(v_1 \otimes \ldots \otimes v_n) = 
  \eta^{\mu\nu} \widetilde{\Omega}_{g',n}( v_1 \otimes \ldots \otimes v_{n-1} \otimes e_{\mu})
  \widetilde{\Omega}_{1,2}(v_n \otimes e_{\nu})_0$$
  where the subscript $0$ stands for the degree-zero part.
  
  If we put $w = e_{\nu}$ in Lemma \ref{degzero}, we get
  $\widetilde{\Omega}_{1,2}(v_n \otimes e_{\nu})_0 = \eta(v_n\cdot e_{\nu}, \alpha)$, thus finally the quantity
  we are looking for is
  $$ \widetilde{\Omega}_{g',n}(  v_1 \otimes \ldots \otimes v_{n-1} \otimes 
  \eta^{\mu\nu}\eta(\alpha \cdot v_n, e_{\mu})e_{\nu} ) 
   = \widetilde{\Omega}_{g',n}( v_1 \otimes \ldots \otimes v_{n-1}\otimes (\alpha \cdot v_n))$$
   
   The general case comes simply by induction noticing that sewing on a genus $g''+1$ surface is the same as
   sewing in a row a genus $g''$ surface and a torus.
 \end{proof}
 
 \begin{proposition}
 \label{sewfreetorus}
 Let $g = g' + g''$ and $\Sigma \in \mathcal{M}_{g'',2}$ be a smooth genus $g''$ surface with two \emph{free} points, 
 and let $$s_{|{\Sigma}} :  {\mathcal{M}}_{g',n}\times\Sigma  \rightarrow \overline{\mathcal{M}}_{g,n}$$
 be the restriction of the sewing map $s$ that sews together
 the points marked $n$ and $2$ respectively, and let $S_{\Sigma}$ be its image.
 With notations as in
 Notation \ref{tubes}, let $\partial N_{\Sigma} = \nu_s^{-1}(S_{\Sigma})$
 and let
 $$\nu_{\Sigma} : \partial N_{\Sigma} \rightarrow S_{\Sigma}$$
 be the restriction of $\nu_s$.
 Then for every $v_i \in A$, we have
 $$\Omega_{g,n}(v_1 \otimes \ldots \otimes v_n) |_{\partial N_{\Sigma}} =
 \nu_{\Sigma}^*\Omega_{g',n}(v_1 \otimes \ldots \otimes v_{n-1} \otimes (\alpha^{g''} \cdot v_n))$$
  
 \end{proposition}
 
 \begin{proof}
 Since the term on the right of the sewing axiom in Definition \ref{freeboundary} is the same as
 the respective one in Definition \ref{fixedboundary} (except for a $\nu_s^*$ that appears), the 
 same calculation of the proof of the previous proposition, together with
 the obvious fact $(\nu_s^*)|_{H^{\bullet}(S_{\Sigma})} = \nu_{\Sigma}^*$, leads to the desired result.
 \end{proof}
 
 \begin{proposition}
 \label{sewnodaltorus}
 Let $g = g' + g''$ and $\Sigma \in \mathcal{M}_{g'',2}$ be a smooth genus $g''$ surface with two \emph{free} points, 
 and let $$s_{|{\Sigma}} : \overline{\mathcal{M}}_{g',n}\times \Sigma  \rightarrow \overline{\mathcal{M}}_{g,n}$$
 be the restriction of the sewing map $s$ that sews together
 the points marked $n$ and $2$ respectively. Then for every $v_i \in A$, we have
 $$ {s_{|{\Sigma}}}^*\overline{\Omega}_{g,n}(v_1 \otimes \ldots \otimes v_n)
  = \overline{\Omega}_{g',n}( v_1 \otimes \ldots \otimes v_{n-1} \otimes (\alpha^{g''} \cdot v_n))$$
 \end{proposition}
 
 \begin{proof}
 The proof is the same as that of Proposition \ref{sewfreetorus}, thanks to the fact that 
 the term in the right-hand side of the sewing axiom in Definition \ref{nodaltheories} is the same as
 in Definition \ref{fixedboundary}.
 \end{proof}
 
 \begin{remark}
 Let us suppose that we want to construct, say, a nodal CohFT $(\overline{\Omega}_{g,n})_{g,n}$,
 and let us suppose we know the theory on $\overline{\mathcal{M}}_{g',n} \times \overline{\mathcal{M}}_{g'',2}$. 
 If the base $(A, \cdot, \mathbf{1}, \eta)$ is \emph{semisimple}, then we can
 recover the theory on $\overline{\mathcal{M}}_{g',n}$. In fact, thanks to Proposition
 \ref{sewnodaltorus} and to Proposition \ref{invertible} that assures that $\alpha$ is invertible
 in a semisimple theory, we have 
 $$ \overline{\Omega}_{g',n}(v_1 \otimes \ldots \otimes v_n) =
 s_{\Sigma}^*\overline{\Omega}_{g,n}(v_1  \otimes \ldots \otimes v_{n-1} \otimes  (\alpha^{-g''} \cdot v_n))$$
 It is important to notice that to calculate the right-hand side of this equality
 we don't need to know the whole homomorphism 
 $\overline{\Omega}_{g,n}: A^{\otimes n} \rightarrow H^{\bullet}(\overline{\mathcal{M}}_{g,n})$,
 but just its restriction to $ \overline{\mathcal{M}}_{g',n} \times \overline{\mathcal{M}}_{g'',2}$,
 because the image of $s_{\Sigma}$ is contained in this subspace.
 
 This will play an important role in the classification of nodal cohomological field theories.
 \end{remark}
 
 \section{Classification}
 
 From now on, all the theories are supposed to be semisimple.
 
 \subsection{Fixed boundary theories}

 Let $\widetilde{\Omega}$ be a fixed boundary CohFT with semisimple base $(A, \mathbf{1}, \eta)$
 and with semisimple basis $(e_{\mu})_{\mu}$.
 We consider $\widetilde{\mathcal{M}}_{g,1}$ embedded in $\widetilde{\mathcal{M}}_{g+1,1}$ by
 the map
 $$(C,[v]) \mapsto \sigma((C,[v]) , (T, [v_1],[v_2])) $$
 where $(T, [v_1], [v_2]) \in \widetilde{\mathcal{M}}_{1,2}$
 is any fixed smooth genus $1$ curve with two marked points,
 $[v]$ and $[v_i]$ are tangent directions at the marked points, and
 $\sigma$ is the sewing-smoothing map of Proposition \ref{smoothing}.
 We call $\sigma_{g,T} : \widetilde{\mathcal{M}}_{g,1} \rightarrow \widetilde{\mathcal{M}}_{g+1,1}$ this embedding.
 
 Therefore, we have a projective system $(H^{\bullet}(\widetilde{\mathcal{M}}_{g,1}))_{g \ge 1}$ with
 projection maps $\sigma_{g,T}^*$. The limit is naturally isomorphic to $\mathbb{C}[\kappa_j]_{j \ge 1}$ by
 Madsen-Weiss theorem:
 $$\underleftarrow{\text{lim}}H^{\bullet}(\widetilde{\mathcal{M}}_{g,1}) \cong \mathbb{C}[\kappa_j]_{j \ge 1}.$$
 
 \begin{proposition}
 The limit $\widetilde{\Omega}^+ = \underleftarrow{\text{lim}} \ \widetilde{\Omega}_{g,1}(\alpha^{-g} \cdot)$ is a 
 well-defined element of $A^* \otimes \mathbb{C}[\kappa_j]_{j \ge 1}$.
 \end{proposition}

 \begin{proof}
  Letting $\Sigma = T$ in Proposition \ref{sewfixedtorus} we find that
  $$\sigma_{g,T}^* \widetilde{\Omega}_{g+1,1}(\alpha^{-g-1}\cdot v) =
  \widetilde{\Omega}_{g,1}(\alpha^{-g-1} \cdot \alpha \cdot v) = \widetilde{\Omega}_{g,1}(\alpha^{-g}\cdot v).$$
  But $\sigma_{g,T}^*$ are the projections in the system $(H^{\bullet}(\widetilde{\mathcal{M}}_{g,1}))_{g \ge 1}$,
  thus the homomorphisms $\widetilde{\Omega}_{g,1}(\alpha^{-g} \cdot)$ are compatible with the projective system.
  Thus, they pass to the limit to get a well-defined homomorphism 
  $\widetilde{\Omega}^+ : A \rightarrow \underleftarrow{\text{lim}}H^{\bullet}(\widetilde{\mathcal{M}}_{g,1})$
  which is an element of $A^* \otimes \mathbb{C}[\kappa_j]_{j \ge 1}$ by Madsen-Weiss theorem.
 \end{proof}
 
 \begin{remark}
 \label{dualproduct}
 Let $i : A \rightarrow A^*$ be the isomorphism of remark \ref{i}. Since $(e_{\mu})_{\mu}$ is an orthonormal basis
 for $\eta$, we have $i(e_\nu) = e^{\nu}$ for each $\nu$, 
 where $(e^{\mu})_{\mu}$ is the dual basis to $(e_{\mu})_{\mu}$. By imposing
 $i$ to be an isomorphism of algebras, we immediately see that $(A^*, \eta^{-1}, i(\mathbf{1}))$ is 
 a \emph{semisimple} Frobenius algebra, with semisimple basis $(e^\mu)_{\mu}$ and same weights as $A$:
 $$ e^\mu \cdot e^\nu = \delta_{\mu,\nu}\theta_\mu^{-1}e^\mu$$
 Thus we can write an explicit formula for the effect of applying a vector $t^\rho e_\rho \in A$ to the
 product of the covectors $v = v_{\mu}e^\mu$ and $w = w_\nu e^\nu$:
 $$((v_{\mu}e^\mu)\cdot (w_\nu e^\nu))(t^\rho e_\rho) =
 v_{\mu}w_{\mu}\theta_{\mu}^{-1}t^\mu = \theta_{\mu}^{-1}t^\mu v(e_{\mu})w(e_{\mu}) $$
 \end{remark}
 
 \begin{remark}
 \label{Groupstructure}
 Since $\widetilde{\Omega}_{g,1} \in A^* \otimes H^{\bullet}(\widetilde{\mathcal{M}}_{g,1})$,
 we can use the Frobenius algebra structure in $A^*$ and the cross product in cohomology to define 
 the product of two homomorphisms $\widetilde{\Omega}_{g,1} \cdot \widetilde{\Omega}_{h,1}$.
 More precisely, thanks to Remark \ref{dualproduct}, we have
 $$\widetilde{\Omega}_{g,1} \cdot \widetilde{\Omega}_{h,1}(v^\mu e_{\mu}) =
 \theta_{\mu}^{-1}v^\mu \widetilde{\Omega}_{g,1}(e_{\mu})\times\widetilde{\Omega}_{h,1}(e_{\mu})$$
 where $\times$ is the usual cross product in the cohomology of the product
 $\widetilde{\mathcal{M}}_{g,1} \times \widetilde{\mathcal{M}}_{h,1}$.
 \end{remark}
 
 \begin{proposition}
 \label{sewingpants}
 Let $\widetilde{C} \in \widetilde{\mathcal{M}}_{0,3}$ be any curve with framed points, and let
 $m_{g,h}: \widetilde{\mathcal{M}}_{g,1} \times \widetilde{\mathcal{M}}_{h,1}
 \rightarrow \widetilde{\mathcal{M}}_{g+h,1}$ be the composition 
 $$ \widetilde{\mathcal{M}}_{g,1} \times \widetilde{\mathcal{M}}_{h,1} \xrightarrow{\sim}
 \widetilde{\mathcal{M}}_{g,1} \times \{ \widetilde{C} \} \times \widetilde{\mathcal{M}}_{h,1}
 \xrightarrow{\sigma_1 \times \text{\emph{Id}}} \widetilde{\mathcal{M}}_{g,2}\times \widetilde{\mathcal{M}}_{h,1}
 \xrightarrow{\sigma_2} \widetilde{\mathcal{M}}_{g+h,1}$$
 where $\sigma_1$ and $\sigma_2$ are the obvious sewing-smoothing maps. Then, in the notation of
 Remark \ref{Groupstructure}, we have
 $$m_{g,h}^*\widetilde{\Omega}_{g+h,1} = \widetilde{\Omega}_{g,1} \cdot \widetilde{\Omega}_{h,1}$$
 
 \end{proposition}
 
 \begin{proof}
 By applying the sewing axiom for fixed boundary theories we get
 $$\sigma_2^*\widetilde{\Omega}_{g+h,1}(v) = 
 \sum_{\mu} \widetilde{\Omega}_{g,2}(v \otimes e_{\mu}) \times \widetilde{\Omega}_{h,1}(e_{\mu}) $$
 and
 $$\sigma_1^*\widetilde{\Omega}_{g,2}(v\otimes e_{\mu}) =
 \sum_\nu \eta(v \cdot e_{\mu}, e_\nu) \widetilde{\Omega}_{g,1}(e_\nu)
 =\eta(v,e_{\mu}\cdot e_{\mu})\widetilde{\Omega}_{g,1}(e_{\mu})$$
 This formula is correct since $\widetilde{\Omega}_{0,3}$, by Definition \ref{fixedboundary},
 has values in the degree zero part of $H^{\bullet}(\widetilde{\mathcal{M}}_{0,3})$, therefore in this case
 we have $\sigma_1^* = (\sigma_1|_{\widetilde{\mathcal{M}}_{g,1} \times \{ C \} })^*$.
 Using the fact that $e_{\mu} \cdot e_{\mu} = \theta_{\mu}^{-1}e_{\mu}$, we get the result thanks to Remark \ref{dualproduct}.
 \end{proof}

 \begin{definition}
  \label{evaluation}
  We write $$i_{g,n}: \mathbb{C}[\kappa_j]_{j \ge 1} \rightarrow H^{\bullet}(\widetilde{\mathcal{M}}_{g,n})$$
  for the $\mathbb{C}$-algebra homomorphism that sends each $\kappa_j$ to the respective $\kappa$-class
  in $\widetilde{\mathcal{M}}_{g,n}$. Notice that $i_{g,1}$ is the projection to the $g$-th factor
  under the identification 
  $\mathbb{C}[\kappa_j]_{j \ge 1} \cong \underleftarrow{\text{lim}} H^{\bullet}(\widetilde{\mathcal{M}}_{g,1})$.
 \end{definition}
  
 \begin{remark}
 \label{infinitypants}
 The maps 
 $$m_{g,h}^*: H^{\bullet}(\widetilde{\mathcal{M}}_{g+h,1}) \rightarrow
 H^{\bullet}(\widetilde{\mathcal{M}}_{g,1}) \otimes H^{\bullet}(\widetilde{\mathcal{M}}_{h,1}) $$
 induced in cohomology by the maps $m_{g,h}$ of Proposition \ref{sewingpants}
 pass to the projective limit giving a map
 $$ m^*:\underleftarrow{\text{lim}} H^{\bullet}(\widetilde{\mathcal{M}}_{g,1}) \rightarrow
 (\underleftarrow{\text{lim}} H^{\bullet}(\widetilde{\mathcal{M}}_{g,1})) \otimes
 (\underleftarrow{\text{lim}} H^{\bullet}(\widetilde{\mathcal{M}}_{g,1}))$$
 that is, by Madsen-Weiss's theorem, a map
 $$m^*: \mathbb{C}[\kappa_j]_{j\ge 1} \rightarrow (\mathbb{C}[\kappa_j]_{j\ge 1}) \otimes
 (\mathbb{C}[\kappa_j]_{j\ge 1})$$
 From now on, by $m^*$ we will mean this map if not otherwise specified.
 
 Of course we have $ m_{g,h}^* \circ i_{g+h,1} = (i_g \otimes i_h) \circ m^*$ by Definition \ref{evaluation}.
 \end{remark}

 \begin{definition}
 \label{infinityproduct}
 Let $X, Y \in A^* \otimes \mathbb{C}[\kappa_j]_{j\ge1}$; we define
 $$X \cdot Y(v^\mu e_{\mu}) = \theta_{\mu}^{-1}v^\mu X(e_{\mu})\otimes Y(e_{\mu})$$
 which is an element of $A^* \otimes (\mathbb{C}[\kappa_j]_{j\ge 1} \otimes \mathbb{C}[\kappa_j]_{j\ge 1})$.
 \end{definition}

With this definition, if $X = (X_g)_{g \ge 1}$ and $Y = (Y_h)_{h \ge 1}$ are projective
systems of homomorphisms $X_g : A \rightarrow H^{\bullet}(\widetilde{\mathcal{M}}_{g,1})$,
then
$$ (i_g \otimes i_h) X \cdot Y = X_g\cdot Y_h$$
 where at the right-hand side we have used the product of Remark \ref{Groupstructure}.

With the identification $\mathbb{C} \otimes \mathbb{C}[\kappa_j]_{j\ge 1} \cong \mathbb{C}[\kappa_j]_{j\ge 1}$, 
we see by an easy computation that the neutral element
of the product is $\theta$, the Frobenius trace of $A$ as in
Definition \ref{Frobeniustrace},
which is indeed a homomorphism of $A$ into $\mathbb{C}\subseteq \mathbb{C}[\kappa_j]_{j \ge 1}$.

 \begin{corollary}
 \label{itisgrouplike}
With the notations of Remark \ref{infinitypants} and Definition \ref{infinityproduct}, 
the limit homomorphism $\widetilde{\Omega}^+$ satisfies 
 $$m^* \widetilde{\Omega}^+ = \widetilde{\Omega}^+ \cdot \widetilde{\Omega}^+$$
 \end{corollary}

 \begin{proof}
 By Proposition \ref{sewingpants}, by Remark \ref{dualproduct}
 and by the identity $\alpha^k = \sum_{\mu}\theta_{\mu}^{-2k+1}e_{\mu}$, we have
 $$m^* \widetilde{\Omega}_{g+h,1}(\alpha^{-g-h}\cdot v) =
 \theta_{\mu}^{2(g+h)-1}v^\mu \widetilde{\Omega}_{g,1}(e_{\mu}) \times
 \widetilde{\Omega}_{h,1}(e_{\mu})= 
 (\widetilde{\Omega}_{g,1}(\alpha^{-g} \cdot) \cdot \widetilde{\Omega}_{h,1}(\alpha^{-h} \cdot))(v)$$
 
 Recalling that $\widetilde{\Omega}^+ = \underleftarrow{\text{lim}} \ \widetilde{\Omega}_{g,1}(\alpha^{-g} \cdot)$
 we  get the result by Remark \ref{infinitypants} and Definition \ref{infinityproduct}.
 \end{proof}
 
 \subsubsection{Properties of $\mathbb{C}[\kappa_j]_{j \ge 1}$}
 
 We want to study more deeply the properties of $\mathbb{C}[\kappa_j]_{j \ge 1}$, in particular
 the map $m^*$ and the product structure of Definition \ref{infinityproduct}.
 To this purpose, we give a definition and prove
 a series of lemmas which will be useful in the sequel.
 
  \begin{lemma}
  \label{firstprimitivekappa}
   With notation as in Remark \ref{infinitypants}, we have 
   $$m^* \kappa_j = \kappa_j \otimes 1 + 1 \otimes \kappa_j$$
   for each $j \ge 1$.
  \end{lemma}

  \begin{proof}
   Let $g$ and $h$ be integers and let us consider the following commutative diagram.
   
   \[
    \begin{CD}
     \widetilde{\mathcal{M}}_{g,2} \times \widetilde{\mathcal{M}}_{h,1} \sqcup \widetilde{\mathcal{M}}_{g,1} 
     \times \widetilde{\mathcal{M}}_{h,2} @>m'_{g,h}>> \widetilde{\mathcal{M}}_{g+h,2} \\
     @V{p'}VV @V{p}VV \\
     \widetilde{\mathcal{M}}_{g,1} \times \widetilde{\mathcal{M}}_{h,1} @>m_{g,h}>>
     \widetilde{\mathcal{M}}_{g+h,1} \\
    \end{CD}
   \]
where $p'$ is the map that forgets the second marked point at each curve, and $m'_{g,h}$ is the map that sews
a fixed element of $\widetilde{\mathcal{M}}_{0,3}$ to the first marked points of each factor.
Then $\kappa_j = p_*(\psi_2^{j+1})$ and $m^*p_*(\psi_2^{j+1}) = p'_*m'^*(\psi_2^{j+1})$
But $m'^*(\psi_2) = \psi_2 \otimes 1 + 1 \otimes \psi_2$ and 
$\psi_2 \otimes \psi_2 = 0 \in H^{\bullet}(\widetilde{\mathcal{M}}_{g,2} 
\times \widetilde{\mathcal{M}}_{h,1} \sqcup \widetilde{\mathcal{M}}_{g,1} \times \widetilde{\mathcal{M}}_{h,2})$.
Therefore $m'(\psi_2^{j+1}) = \psi_2^{j+1} \otimes 1 + 1 \otimes \psi_2^{j+1}$ and 
$p'_*m'^*(\psi_2^{j+1}) = \kappa_j \otimes 1 + 1 \otimes \kappa_j$. Taking the limit in $g$ and $h$ we get the
statement, thanks to Remark \ref{infinitypants}.
  \end{proof}

 \begin{remark}
  \label{Hopfalgebra}
  The last lemma says that the polynomial ring $\mathbb{C}[\kappa_j]_{j \ge 1}$ is a Hopf algebra,
  the co-product being $m^*$, and the antipode map being
  $S: \kappa_j \mapsto -\kappa_j$ for each $j$ and extended to a $\mathbb{C}$-linear homomorphism
  $\mathbb{C}[\kappa_j]_{j \ge 1} \rightarrow \mathbb{C}[\kappa_j]_{j \ge 1}$.

  If $A$ is a Frobenius algebra, it is in fact a bialgebra (but not necessarily a Hopf algebra) via the co-product
  $$\Delta(v) = \eta(v, e_{\mu}\cdot e_{\nu})e_{\mu}\otimes e_{\nu}$$
  which is such that, for every $w,t \in A$,
  $$\eta \otimes \eta (\Delta(v), e_{\mu}\otimes e_{\nu}) = \eta(v, e_\mu\cdot e_\mu).$$
  Notice that, in the notation of Definition \ref{infinityproduct}, we have
  $$X \cdot Y = (X \otimes Y)\circ \Delta.$$
  In the semisimple case, the co-product simplifies to $\Delta(v) = \theta_\mu^{-1}v^\mu e_\mu \otimes e_\mu$.
  Let us stick to the semisimple case.
  Taking $P(v) = v^\mu \theta_\mu^2 e_\mu$, we see that $\Delta$ has a left inverse given by 
  $\cdot \circ (\text{Id} \otimes P)$, where we have denoted by $\cdot : A\otimes A \rightarrow A$ the product
  of the Frobenius algebra $A$.
  The endomorphism $P$ of $A$ defines a co-product $\Delta^*$ on $A^*$
  that makes it a bialgebra via the formula
  $$\Delta^*(\phi)(v \otimes w) = \phi(P(v)\cdot P(w)).$$
  Notice that with this definition the isomorphism $i: A \rightarrow A^*$
  of Remark \ref{i} is \emph{not} an isomorphism
  of bialgebras.
  
  Thus, $A^* \otimes \mathbb{C}[\kappa_j]_{j \ge 1}$ is a bialgebra: the
  co-product $\widetilde{\Delta}$ being
  the tensor product of the two co-products on each factor. By the identification
  $A^* \otimes \mathbb{C}[\kappa_j]_{j\ge 1} \cong \text{Hom}_{\mathbb{C}}(A, \mathbb{C}[\kappa_j]_{j \ge 1})$,
  we see that the following diagram commutes.
\[
\begin{CD}
A  @>f>> \mathbb{C}[\kappa_j]_{j \ge 1} \\
@V{\Delta}VV @Vm^* VV \\
A \otimes A @>>\widetilde{\Delta} f> \mathbb{C}[\kappa_j]_{j \ge 1} \otimes \mathbb{C}[\kappa_j]_{j \ge 1} \\
\end{CD}
\]

 \end{remark}

 \begin{definition}
 Let $y \in \mathbb{C}[\kappa_j]_{j \ge 1}$. Then $y$ is \emph{primitive} if 
 $$ m^*y = y \otimes 1 + 1 \otimes y$$
 with the notation of Remark \ref{infinitypants}.
 \end{definition}

  \begin{lemma}
 \label{kappaprimitive}
 A class $y \in \mathbb{C}[\kappa_j]_{j \ge 1}$ is primitive if and only if it is a linear combination
 of $\kappa$ classes.
 \end{lemma}
 
 \begin{proof}
 Each $\kappa_j$ is primitive by Lemma \ref{firstprimitivekappa}. Since the primitive classes clearly form
 a $\mathbb{C}$-vector space, we deduce that every linear combination of $\kappa$ classes is primitive.
 
 Conversely, Milnor-Moore's theorem says that a Hopf algebra is a free algebra on its primitive classes.
 Since $\mathbb{C}[\kappa_j]_{j \ge 1}$ is a free algebra on the $\kappa_j$'s, and since they are primitive,
 there cannot be any other primitive class. This completes the proof.
 \end{proof}

 \begin{definition}
 Let $X: A \rightarrow \mathbb{C}[\kappa_j]_{j \ge 1}$ be a homomorphism. $X$ is \emph{group-like}
 if $X_0 = \theta$, the Frobenius trace of $A$, and it satisfies the equality
 $$m^*X = X \cdot X$$
 $X$ is \emph{primitive} if it satisfies the equality
 $$m^*X = X \otimes 1 + 1 \otimes X$$
 \end{definition}
 
 \begin{remark}
 \label{kappaprimitive1}
 Lemma \ref{kappaprimitive} says that a homomorphism $x: A \rightarrow \mathbb{C}[\kappa_j]_{j \ge 1}$ is primitive
 if and only if there are covectors $\phi_j \in A^*$ such that $x = \sum_j \phi_j\kappa_j$.
 \end{remark}
 
 \begin{lemma}
 \label{exponential}
 Let $x: A \rightarrow \mathbb{C}[\kappa_j]_{j \ge 1}$ be a homomorphism, and let
 $\text{\emph{exp}}(x)$ be the homomorphism defined by the infinite series
 $$\text{\emph{exp}}(x) = \sum_{n=0}^{\infty} \frac{x^{\cdot n}}{n!}$$
 where $x^{\cdot n} :A \rightarrow \mathbb{C}[\kappa_j]_{j \ge 1}$ is the $n$-th power for the product
 of Definition \ref{infinityproduct} composed with the product of polynomials
 $\mathbb{C}[\kappa_j]_{j \ge 1}^{\otimes n} \rightarrow \mathbb{C}[\kappa_j]_{j \ge 1}$,
 with the convention $x^{0} = \theta$ the Frobenius trace of $A$ as in Definition \ref{Frobeniustrace}.
 Then we have 
 $$m^*\text{\emph{exp}}(x) = \text{\emph{exp}}(m^*x)$$
 \end{lemma}
 
 \begin{proof}
 By linearity we have
 $$m^*\text{exp}(x) = \sum_n \frac{m^*(x^{\cdot n})}{n!}$$
 therefore we have to calculate $m^*(x^{\cdot n})$. By applying Definition \ref{infinityproduct}
  with $X = x$, $Y = x^{n-1}$, we see that
 $$x^{\cdot n}(v^\mu e_{\mu}) = \theta_{\mu}^{-n+1}v^\mu (x(e_{\mu}))^{ n}$$
 Where now $(x(e_{\mu}))^{ n}$ is the usual power of polynomials.
 
 Now we have $m^*[x(e_{\mu})^n] = [m^*(x(e_{\mu}))]^n$, thus
$$m^*(x^{\cdot n}) = (m^*x)^{\cdot n} $$ which immediately yields our result.
 \end{proof}

 \begin{lemma}
 \label{grouplikeprimitive}
 Let $X: A \rightarrow \mathbb{C}[\kappa_j]_{j \ge 1}$ be a homomorphism. Then $X$ is group-like
 if and only if there exists a primitive homomorphism $x$ such that $X = \text{\emph{exp}}(x)$.
 \end{lemma}

 \begin{proof}
 Let $x$ be primitive, and let $X = \text{exp}(x)$. Then by Lemma \ref{exponential} we have
 $m^*X = \text{exp}(m^*x) = \text{exp}(x\otimes 1 + 1 \otimes x)$. Moreover, expanding out the powers, we get
 $$ \text{exp}(x \otimes 1 + 1 \otimes x) = \sum_{n=0}^{\infty}\sum_{k_1 + k_2=n} \frac{1}{k_1!k_2!}
 (x\otimes 1)^{\cdot k_1} \cdot (1 \otimes x)^{\cdot k_2}= \text{exp}(x \otimes 1) \cdot \text{exp}(1 \otimes x)$$
 Now we clearly have $(x \otimes 1)^{\cdot n} = x^{\cdot n} \otimes 1$, therefore the quantity in the
 above expression is $(X \otimes 1) \cdot (1 \otimes X)$, and this is in turn $X \cdot X$ since by Remark
 \ref{infinityproduct} we have
 $$(X \otimes 1) \cdot (1 \otimes X)(v) = \theta_{\mu}^{-1}v^\mu (X(e_\mu)\otimes 1)(1 \otimes X(e_{\mu})) =
 \theta_{\mu}^{-1}v^\mu X(e_{\mu})X(e_{\mu}) = X\cdot X (v)$$
 Thus we see that $X$ is group-like.
 
 Conversely, let $X$ be group-like and
  let us define
 $$ x = \text{log}(X) := \sum_{n=1}^{\infty}(-1)^{n-1}\frac{(X-\theta)^{\cdot n}}{ n}$$
 then $x$ is well-defined since
 the sum at the right-hand side is finite for each degree, and
 formally $\text{exp}(x) = X$. Moreover, the technique used in the proof of Lemma \ref{exponential} yields
 $m^*x = \text{log}(m^*X)$, therefore using the fact that $X$ is group-like by hypothesis we get
 $$m^*X = \text{log}(X \cdot X)$$
 Writing $X \cdot X = (X \otimes 1)\cdot(1\otimes X)$ and $(X-\theta)^{\cdot n} \otimes 1
 = ((X-\theta)\otimes 1)^{\cdot n}$
 in the expression for $\text{log}(X)$, we get the expression
 $$\text{log}(X\cdot X) = \text{log}(X) \otimes 1 + 1 \otimes \text{log}(X)$$
 from which we see that $x$ is primitive. This completes the proof.
 \end{proof}
 
 \subsubsection{Classification of smooth theories}
 
 We are now ready to classify fixed boundary theories.
 
 \begin{proposition}
 \label{howisomegalike}
  There exist covectors $\phi_j \in A^*$ such that
  $ \widetilde{\Omega}^+ = \text{\emph{exp}}(\sum_{j>0} \phi_j \kappa_j)$.
 \end{proposition}

 \begin{proof}
 By Corollary \ref{itisgrouplike} we know that $\widetilde{\Omega}^+$ is group-like. By Lemma
 \ref{grouplikeprimitive}, there exists a primitive homomorphism $x: A \rightarrow \mathbb{C}[\kappa_j]_{j \ge 1}$
 such that $\widetilde{\Omega}^+ = \text{exp}(x)$. By Remark \ref{kappaprimitive1} this $x$ is an $A^*$-linear
 combination of $\kappa_j$ classes and this yields our result.
 \end{proof}
 
 Let us now see how to recover the homomorphisms $\widetilde{\Omega}_{g,n}$ from $\widetilde{\Omega}^+$.
 For this, we have the first classification theorem.
 
  \begin{theorem}
  \label{classificationfixed}
  Let $\widetilde{\Omega}_{g,n} : A^{\otimes n} \rightarrow H^{\bullet}(\widetilde{\mathcal{M}}_{g,n})$ be a
  fixed boundary CohFT. Then there exist covectors $\phi_j \in A^*$ for $j \ge 1$ such that, setting 
  $\widetilde{\Omega}^+ = \text{\emph{exp}}(\sum_{j\ge 1} \phi_j \kappa_j) :
  A \rightarrow \mathbb{C}[\kappa_j]_{j \ge 1} $, we have
  \begin{equation}
  \label{strangeformula}
   \widetilde{\Omega}_{g,n}(v_1 \otimes \cdots \otimes v_n) = 
  i_{g,n}\widetilde{\Omega}^+(\alpha^g \cdot v_1 \cdots v_n)
  \end{equation}
  for $g \ge 1$.
  Conversely, if a semisimple, symmetric, commutative Frobenius algebra structure is defined on
  $(A, \cdot, \mathbf{1}, \eta )$, then for every choice of covectors $\phi_j \in A^*$ for $j \ge 1$, we get a 
  fixed boundary CohFT using the above formulas for $g \ge 1$, and
  $$\widetilde{\Omega}_{0,n}(v_1 \otimes \ldots \otimes v_n) =
  s^*\widetilde{\Omega}_{1,n}(v_1 \otimes \ldots \otimes (\alpha^{-1} \cdot v_n))$$
  where $s: \widetilde{\mathcal{M}}_{0,n} \times \Sigma \rightarrow \widetilde{\mathcal{M}}_{1,n}$
  for any $\Sigma \in \widetilde{\mathcal{M}}_{1,2}$.
  
 \end{theorem}

 \begin{proof}
 
 From the definition 
 $\widetilde{\Omega}^+ = \underleftarrow{\text{lim}}\widetilde{\Omega}_{g,1}(\alpha^{-g} \cdot)$ we 
 immediately find $\widetilde{\Omega}_{g,1}(v) = i_{g,1}\widetilde{\Omega}^+(\alpha^g \cdot v)$.
 Let
 $\varphi_{g,n} : \widetilde{\mathcal{M}}_{g,1} \rightarrow \widetilde{\mathcal{M}}_{g,n}$
 for $n \ge 2$ be the sewing-smoothing map applied to
 a fixed element of $\widetilde{\mathcal{M}}_{0,n+2}$, and let 
 $s_{g,n}: \widetilde{\mathcal{M}}_{g,n} \rightarrow \widetilde{\mathcal{M}}_{g+g',n}$
 be the sewing-smoothing map applied to a fixed element of $\widetilde{\mathcal{M}}_{g',2}$
 sewed to $\widetilde{\mathcal{M}}_{g,n}$.
 It is easy to see by induction on $n$ that
 \begin{equation}
 \label{multiproduct}
 \varphi_{g,n}^* \widetilde{\Omega}_{g,n}(v_1 \otimes \cdots \otimes v_n) =
 \widetilde{\Omega}_{g,1}(v_1 \cdots v_n)
 = i_{g,1}\widetilde{\Omega}^+(\alpha^g \cdot v_1 \cdots v_n)
 \end{equation}
 and by Proposition \ref{sewfixedtorus} 
 we have
 \begin{equation}
 \label{multiproduct2}
  \widetilde{\Omega}_{g,n}( v_1 \otimes \cdots \otimes  v_n) =
 s_{g,n}^*\widetilde{\Omega}_{g+g',n}(v_1 \otimes \cdots \otimes (\alpha^{-g'} \cdot v_n)) 
 \end{equation}
 Harer's stability theorem says that the map $\varphi_{g+g',n}$
 is a homology equivalence in degree less than $(g+g')/3$,
 thus if this quantity is greater than $3g - 3 + n$, we
 deduce by (\ref{multiproduct}) and (\ref{multiproduct2}) the following formula
  \begin{equation}
  \label{quasifinale}
  \widetilde{\Omega}_{g,n}(v_1 \otimes \cdots \otimes v_n) = 
  s_{g,n}^*(\varphi_{g+g',n}^*)^{-1}i_{g+g',1}\widetilde{\Omega}^+(\alpha^g \cdot v_1 \cdots v_n)
  \end{equation}
 which is well-defined since the degree of the left-hand side
 is in the range in which $\varphi_{g+g',n}$ is an isomorphism. Now, in the stable range, we know by
 Harer's theorem that the inverse of $\varphi_{g+g',n}$ is exactly $p$, the forgetful map. Moreover
 it is clear that
 $p \circ s_{g,n} = s_{g,1} \circ p$
 and that 
$s_{g,1}^*i_{g+g',1} = i_{g,1}$. Substituting in (\ref{quasifinale}) we get the desired formula.
 Thus we have shown the first part.
 
  For the second part, we need to verify the axioms. Let then $(e_{\mu})_{\mu}$ be a semisimple basis
  for the Frobenius algebra $(A, \cdot, \mathbf{1} ,\eta)$.
  
  Since $(A, \cdot)$ is commutative, $\widetilde{\Omega}_{g,n}$ is $S_n$-invariant, so the first axiom is verified.
  
  For the second axiom, we have
  $$\widetilde{\Omega}_{0,3}(v_1 \otimes v_2 \otimes \mathbf{1}) =
  s^*\widetilde{\Omega}_{1,3}(v_1 \otimes v_2 \otimes (\alpha^{-1} \cdot \mathbf{1})) =
  s^*p^*i_1^*\widetilde{\Omega}^+(v_1 \cdot v_2)$$
  Now from the definition of $\widetilde{\Omega}^+$ we see that $i_1^*\widetilde{\Omega}^+(v_1 \cdot v_2) =
  \theta(v_1 \cdot v_2) + \phi_1(v_1 \cdot v_2)\kappa_1$,
  where $\theta$ is the Frobenius trace of $A$ (the classes $\kappa_j$ for $j \ge 2$ are zero in the ring
  $H^{\bullet}(\widetilde{\mathcal{M}}_{1,1})$ to which $i_1^*\widetilde{\Omega}^+(v_1 \cdot v_2)$ belongs);
  moreover, we have $p^*\kappa_1 = \kappa_1 - \psi_2 - \psi_3$, thus $s^*p^*\kappa_1 = 0$ and finally
  $$s^*p^*i_1^*\widetilde{\Omega}^+(v_1 \cdot v_2) = \theta(v_1 \cdot v_2) = \eta(v_1, v_2) $$
  and the axiom is verified.
  
  For the third one, let $\tau: \widetilde{\mathcal{M}}_{g-1,n+2} \rightarrow \widetilde{\mathcal{M}}_{g,n}$
be the non-separating sewing-smoothing map. Then
$$\tau^*\widetilde{\Omega}_{g,n}(v_1 \otimes \ldots \otimes v_n) = 
\tau^*p^*i_g^*\widetilde{\Omega}^+(\alpha^g \cdot v_1 \cdots v_n)$$
Now $p \circ \tau = \tau \circ p_{1,n+1,n+2}$ 
where $p_{1,n+1,n+2}:\widetilde{\mathcal{M}}_{g-1,n+2} \rightarrow \widetilde{\mathcal{M}}_{g-1,3}$
is the map forgetting all the points except
the ones marked $1$, $n+1$ and $n+2$, therefore we must compute
$\tau^*i_g^*\widetilde{\Omega}^+(\alpha^g \cdot v_1 \cdots v_n)$
where $\tau: \widetilde{\mathcal{M}}_{g-1,3} \rightarrow \widetilde{\mathcal{M}}_{g,1}$.
Now $\tau^*\psi_i = \psi_i$ for every $i$, therefore 
$\tau^*\kappa_i = \kappa_i$ for every $i$ and we simply have 
$$\tau^*\iota_g^*\widetilde{\Omega}^+(\alpha^g \cdot v_1 \cdots v_n) = 
p_{2,3}^*i_{g-1}^*\widetilde{\Omega}^+(\alpha^g \cdot v_1 \cdots v_n) $$
Putting all this together, we finally have
$$\tau^*\widetilde{\Omega}_{g,n}(v_1 \otimes \ldots \otimes v_n) =
p^*i_{g-1}^*\widetilde{\Omega}^+(\alpha^g \cdot v_1 \cdots v_n)=$$
$$=p^*i_{g-1}^*\sum_{\mu}\widetilde{\Omega}^+(\alpha^{g-1} \cdot v_1 \cdots \cdot v_n \cdot e_{\mu} \cdot e_{\mu}) =
\sum_{\mu}\widetilde{\Omega}_{g,n}(v_1 \otimes \ldots \otimes v_n \otimes e_{\mu} \otimes e_{\mu})$$
taking into account that $e_{\mu}$ is a semisimple basis, we have verified the third axiom.
The fourth axiom is proved similarly, while the fifth axiom is obvious.
 \end{proof}

 \subsection{Free boundary theories}
 \label{freeboundarytheories}
 
 Now that we have proved the classification theorem for fixed boundary theories, we proceed to study
  free boundary theories. Thus, let $\Omega$ be a semisimple
 free boundary CohFT with semisimple base $(A, \cdot, \mathbf{1}, \eta)$ with semisimple basis $(e_{\mu})_{\mu}$.
 We recall that if $\pi: \widetilde{\mathcal{M}}_{g,n} \rightarrow \mathcal{M}_{g,n}$ is the map
 that forgets the tangent directions, then Proposition \ref{coherfree} says that 
 $\widetilde{\Omega} = \pi^*\Omega$ is 
 a fixed boundary CohFT, with the same base. Therefore, by Theorem \ref{classificationfixed}, there
 exists a homomorphism $\widetilde{\Omega}^+: A \rightarrow \mathbb{C}[\kappa_j]_{j \ge 1}$ such that
 $$ \pi^* \Omega_{g,n}(v_1 \otimes \ldots \otimes v_n) =
 i_{g,n} \widetilde{\Omega}^+(\alpha^g \cdot v_1 \cdots v_n)$$
 
 Let $\pi^{\{ 2 \}} : \widetilde{\mathcal{M}}_{g,2}^{\{ 2 \} } \rightarrow \widetilde{\mathcal{M}}_{g,2}$
 be the bundle map like in Notation \ref{torusbundle}, let 
 $s|_{\Sigma}^{\{ 2 \} }=\pi^{\{ 2 \}*}s|_{\Sigma}$ with notation as in Proposition \ref{sewfixedtorus},
 and let $\Omega_{g,2}^{\{ 2 \}} = \pi^{\{ 2 \}*}\Omega_{g,2}$.
 Then, by Proposition \ref{sewfixedtorus}, we have
 $$s|_{\Sigma}^{\{ 2 \} *}\Omega_{g,2}^{\{ 2 \}}(v_1 \otimes v_2) =
 \Omega_{g',2}^{\{ 2 \}}(v_1 \otimes (\alpha^{g'} \cdot v_2)) $$
 therefore we have a projective system $\Omega_{g,2}^{\{ 2 \}}(\cdot \otimes (\alpha^{-g} \cdot))$.
 We denote the limit object
 $$\Omega^+ = \underleftarrow{\text{lim}}\Omega_{g,2}^{\{ 2 \}}(\cdot \otimes (\alpha^{-g} \cdot)):
 A \otimes A \rightarrow \mathbb{C}[\psi,\kappa_j]_{j \ge 1}$$
 where $\psi$ is the $\psi$-class of the only free point of $\widetilde{\mathcal{M}}_{g,2}^{\{ 2 \}}$.
 We can see $\Omega^+$ as an element of $(A^*\otimes A^*)[\kappa, \psi_j]_{j\ge 1}$, therefore composing with
 the isomorphism $i^{-1}: A^* \rightarrow A$ of Remark \ref{i} we get an element
 $Z(\kappa, \psi_j)_{j\ge1} \in (A^* \otimes A)[[\kappa, \psi_j]]_{j\ge 1}$. 
 Recalling that there is a natural isomorphism $A^* \otimes A \simeq \text{End}(A)$, we define
 $$R(\psi) = Z(0,-\psi)^*  \in \text{End}(A)[[\psi]]$$
 where the star indicates the adjoint endomorphism with respect to $\eta$.
 
 \begin{remark}
 \label{tounderstand}
 
 This definition simply means that for every $v, w \in A$, we have
 \begin{equation}
 \label{Zmorphism}
 \eta(Z(\kappa, \psi)v, w) = \Omega^+(v\otimes w)(\kappa, \psi)
 \end{equation}
where $\Omega^+(v\otimes w)(\kappa, \psi_j)_{j \ge 1} \in \mathbb{C}[\kappa, \psi_j]_{j \ge 1}$ and $\eta$ at the
left-hand side is applied coefficient-wise, that is, $\eta((\sum_if_i\psi^i)v, w) = \sum_i\eta(f_i(v), w) \psi^i$.
By Lemma \ref{degzero}, the degree-zero part 
$\Omega_{1,2}^{\{ 2 \}}(v \otimes \alpha^{-1}\cdot w)_0$ is $\eta(v \cdot \alpha^{-1} \cdot w, \alpha) = \eta(v,w)$;
therefore it is the same for $\Omega_{g,2}^{\{ 2 \}}(v \otimes \alpha^{-g}\cdot v)_0$ for every $g \ge 1$,
thus we see 
$$\eta(v, R(-\psi)_0w) = \Omega^+(v\otimes w)(\psi, 0)_0 = \eta(v,w)$$
for every $v, w \in A$, and we deduce $R(\psi)_0 = \text{Id}$. From now on, we will write
$$ R(\psi) = \text{Id} + R_1 \psi + R_2 \psi^2 + \ldots$$
for $R_i \in \text{End}(A)$. In particular, $R(\psi)$ is invertible.
 \end{remark}
 
 \begin{lemma}
 \label{firstformula}
 We have
 $$\widetilde{\Omega}^+(v_1 \cdot v_2)(\kappa) = \eta( Z(\kappa, \psi)v_1,R(-\psi)^{-1}v_2) = 
 \Omega^+(v_1 \otimes R(-\psi)^{-1}v_2)$$
 \end{lemma}
 
 \begin{proof}
 Throughout this proof, we refer to the notation used in chapter 3, in particular
 to Notation \ref{tubeneighfree}. The sewing axiom for free boundary CohFTs says that
 $$\Omega_{g,2}(\alpha^{-g_1} \cdot v_1 \otimes \alpha^{-g_2} \cdot v_2)|_{\partial N_s} = 
 \nu_{g_1,g_2}^*\sum_{\mu}\Omega_{g_1,2}(e_{\mu} \otimes \alpha^{-g_1} \cdot v_1) 
 \times \Omega_{g_2, 2}(e_{\mu} \otimes \alpha^{-g_2} \cdot v_2)$$
 where 
 $$\nu_{g_1,g_2}: \partial N_s \simeq \widetilde{\mathcal{M}}_{g_1,2}^{\{ 2 \}} \times_{S^1}
 \widetilde{\mathcal{M}}_{g_2,2}^{\{ 2 \}} \rightarrow
 {\mathcal{M}}_{g_1,2} \times
 {\mathcal{M}}_{g_2,2}$$
 (just for this time, we have sewed the points marked $1$ on each curve, the sewing axiom is therefore
 slightly modified). If $\pi: \widetilde{\mathcal{M}}_{g,n} \rightarrow \mathcal{M}_{g,n}$ is the bundle map
 forgetting the tangent directions, we get
 $$\widetilde{\Omega}_{g,2}(\alpha^{-g_1}\cdot v_1 \otimes \alpha^{-g_2} \cdot v_2)|_{\pi^{-1}(\partial N_s)} =
 {\nu}_{g_1,g_2}^*\sum_{\mu} \Omega_{g_1,2}^{\{2\}}(e_{\mu} \otimes \alpha^{-g_1} \cdot v_1) \times
 \Omega_{g_2,2}^{\{2\}}(e_{\mu} \otimes \alpha^{-g_2} \cdot v_2)$$
 By passing to the limit in $g_1$ and $g_2$ we get
 $$\widetilde{\Omega}^+(v_1\cdot v_2)(\kappa) = \nu^*\sum_{\mu}\Omega^+(e_{\mu} \otimes v_1)(\kappa',\psi')
 \Omega^+(e_{\mu}\otimes v_2)(\kappa'',\psi'')$$
 where the left-hand side is computed with Theorem \ref{classificationfixed}.
 Here $\kappa'$ and $\psi'$ refer to the respective classes in $\widetilde{\mathcal{M}}_{g_1,2}^{\{2\}}$,
 $\kappa''$ and $\psi''$ in $\widetilde{\mathcal{M}}_{g_2,2}^{\{2\}}$. Then the $\kappa_i'$'s and the
 $\kappa_i''$'s are algebraically independent. We set
 $\kappa = \kappa' + \kappa''$
 and $\nu^*$ to be the limit of the maps ${\nu}_{g_1,g_2}^*$. Since for every $g_1$, $g_2$ we have
 $\nu_{g_1, g_2}^* \psi' = -\nu_{g_1,g_2}^*\psi''$, and $\nu_{g_1,g_2}^*$ sends the other classes
 to the respective classes on $\mathcal{M}_{g,2}$, we can write
 \begin{equation}
 \label{auxiliaryequation}
 \widetilde{\Omega}^+(v_1\cdot v_2)(\kappa) = \sum_{\mu}\Omega^+(e_{\mu} \otimes v_1)(\kappa',-\psi)
 \Omega^+(e_{\mu}\otimes v_2)(\kappa'',\psi)
 \end{equation}
 therefore, setting $\kappa'=0$, $\kappa'' = \kappa$ in (\ref{auxiliaryequation})
 and using Equation (\ref{Zmorphism}), we get
 \begin{equation}
 \label{almostlast}
 \widetilde{\Omega}^+(v_1\cdot v_2)(\kappa) = \sum_{\mu}\eta(R(\psi)v_1, e_{\mu})\eta(v_2, Z(\kappa,\psi)e_{\mu}) =
 \eta(R(\psi)v_1, Z(\kappa,\psi)^*v_2).
 \end{equation}
 Lemma \ref{symplecticcondition} below implies that $Z(\kappa, \psi)$
 commutes with its adjoint, therefore the last term
 in Equation (\ref{almostlast}) is $\eta( Z(\kappa,\psi)v_1,R(\psi)^*v_2)$.
 This is the first equality again by Lemma \ref{symplecticcondition} below,
 and the second one is simply Equation (\ref{Zmorphism}).
 \end{proof}

 \begin{lemma}
 \label{symplecticcondition}
 The endomorphism $R(\psi)$ satisfies the \emph{symplectic condition}, that is
 $$R(\psi)^* = R(-\psi)^{-1}$$
 Where $R(\psi)^*$ is the adjoint of $R(\psi)$ with respect to the bilinear form $\eta$.
 \end{lemma}
 
 \begin{proof}
  Setting $\kappa = 0$ in Equation (\ref{almostlast})
 we have 
 $$\eta(R(\psi)v_1, R(-\psi)v_2) = \eta(v_1, v_2)$$
 that is, $R(\psi)^* = R(-\psi)^{-1}$ since $R(-\psi)$ is invertible by Remark \ref{tounderstand}.
 \end{proof}

 \begin{corollary}
 \label{betterformula}
 We have
 $$\widetilde{\Omega}^+(v_1 \cdot v_2) = \Omega^+(R(\psi)v_1 \otimes v_2)$$
 \end{corollary}
 
 \begin{proof}
 By Lemma \ref{firstformula} we have
 $$\widetilde{\Omega}^+(v_1 \cdot v_2)(\kappa) = \eta(R(\psi)Z(\kappa, \psi)v_1, v_2) $$
 we now get the corollary by the fact that
 $Z(\kappa, \psi)$ and $R(\psi)$ commute, and by Equation \ref{Zmorphism}. 
 \end{proof}

 Similarly to the previous subsection, we want to recover $\Omega_{g,n}$ from the limit object $\Omega^+$.
 
 \begin{theorem}
 \label{classificationfree}
 Let $i_{g,n} : \mathbb{C}[\kappa_j]_{j\ge 1} \rightarrow H^{\bullet}(\mathcal{M}_{g,n})$ be the $\mathbb{C}$-linear
 map that sends each $\kappa$-class to the respective class in $\mathcal{M}_{g,n}$. Then we have
 $$\Omega_{g,n}(v_1 \otimes \ldots \otimes v_n) = 
 i_{g,n}\widetilde{\Omega}^+(\alpha^g \cdot  R(\psi_1)^{-1}v_1 \cdot \ldots \cdot R(\psi_n)^{-1}v_n)(\kappa)$$
 \end{theorem}
 
 \begin{proof}
 Let us consider a multi-sewing map
 $$s: \mathcal{M}_{g,n} \times 
 (\widetilde{\mathcal{M}}^{\{2\}}_{g_1,2} \times \cdots \times \widetilde{\mathcal{M}}^{\{2\}}_{g_n,2}) \rightarrow
 \widetilde{\overline{\mathcal{M}}}_{g+G, n}$$
 that sews the $k$-th marked point of $\mathcal{M}_{g,n}$ to the first marked point of 
 $\mathcal{M}^{\{2\}}_{g_k,2}$ for 
 each $k$. Here $G = g_1 + \ldots +g_n$. This map is nothing but the composition, for $1 \le i \le k$, of the maps
 $s_i: \widetilde{\overline{\mathcal{M}}}^{\{1,\ldots, i-1\}}_{g+g_1 +\ldots + g_{i-1},n}\times \mathcal{M}_{g_i, 2}
 \rightarrow \widetilde{\overline{\mathcal{M}}}^{\{1,\ldots,i\}}_{g+g_1 + \ldots + g_i,2}$ 
 that sew the $i$-th marked point of the first curve to the first point of the second curve.
 Now let us apply the sewing axiom for free boundary CohFTs $n$ times, and
 multiply each $i$-th entry by $\alpha^{-g_i}$; we get
 $$\Omega_{g+G,n}((\alpha^{-g_1}\cdot v_1) \otimes \ldots \otimes (\alpha^{-g_n} \cdot v_n))|_{\partial N_s} = $$
 $$ =\nu^*\sum_{\mu_1, \ldots, \mu_n}\Omega_{g,n}(e_{\mu_1} \otimes  \ldots \otimes e_{\mu_n}) 
 \times \Omega_{g_1, 2}^{\{2\}}(e_{\mu_1} \otimes  (\alpha^{-g_1} \cdot v_1)) \times 
 \cdots \times \Omega_{g_n, 2}^{\{2\}}(e_{\mu_n} \otimes  (\alpha^{-g_n} \cdot v_n))$$
 where now $\partial N_s$ is the circular neighbourhood, given by the circular neighbourhood theorem,
 of the image $S$ of $s$, and $\nu: \partial N_s \rightarrow S$ is the respective circle bundle.
 
 Now let $g_1, \ldots, g_n$ go to infinity, and use Lemma \ref{firstformula}
 to compute the factors at the right-hand side of the equation. Then we get
 $$ \widetilde{\Omega}^+(\alpha^g \cdot v_1 \cdots v_n) =  
 \sum_{\mu_1, \ldots, \mu_n}\Omega_{g,n}(e_{\mu_1} \otimes  \ldots \otimes e_{\mu_n}) 
 \eta(Z(\kappa, -\psi_1)e_{\mu_1},v_1 ) \cdots
 \eta(Z(\kappa, -\psi_n)e_{\mu_n},v_n )$$
 where the $\psi_i$'s are the $\psi$-classes attached to the marked points of $\mathcal{M}_{g,n}$
 (therefore $\nu^*\psi'_i = -\psi_i$ if $\psi'_i$ is attached to the free point of
 $\widetilde{\mathcal{M}}_{g_1,2}^{\{2\}}$). Let  $\kappa^{(g,n)}$ denote the $\kappa$-classes on
 $\mathcal{M}_{g,n}$, then we have
 $$\widetilde{\Omega}^+(\alpha^g \cdot v_1 \cdots v_n)(\kappa^{(g,n)} + \kappa) = 
 \Omega_{g,n}(Z(\kappa^{(1)}, -\psi_1)^*v_1 \otimes
 \ldots \otimes Z(\kappa^{(n)}, -\psi_n)^*v_n)$$
 where $\kappa = \kappa^{(1)} + \ldots + \kappa^{(n)}$. Setting one by one each $\kappa^{(i)} = 0$  we get
  $$\Omega_{g,n}(v_1 \otimes \ldots \otimes v_n) = 
  \widetilde{\Omega}^+(\alpha^g \cdot R(\psi_1)^{-1}v_1 \cdot \ldots \cdot R(\psi_n)^{-1}v_n)(\kappa^{(g,n)}) = $$
  $$ =
  i_{g,n}\widetilde{\Omega}^+(\alpha^g \cdot R(\psi_1)^{-1}v_1 \cdot \ldots \cdot R(\psi_n)^{-1}v_n)(\kappa)$$
  which is the statement of the theorem.
  \end{proof} 
 
 This theorem, in turn, allows us to prove a relation that must exist between $\widetilde{\Omega}^+$ and
 $R(\psi)$.
 
 \begin{proposition}
 \label{hereisthelogarithm}
 Let $A$ be a symmetric and commutative Frobenius algebra.
  Let $\widetilde{\Omega}^+: A \rightarrow \mathbb{C}[\kappa_j]_{j \ge 1}$ be group-like, and let
  $R(\psi) \in \text{\emph{End}}(A)[[\psi]]$ be such that $R(0) = \text{\emph{Id}}$
  and $R(\psi)^* = R(-\psi)^{-1}$.
  Then the formulas of Theorem \ref{classificationfree} 
  define a free boundary CohFT if and only if for every $v \in A$, 
  $$ \text{\emph{log}}\widetilde{\Omega}^+(v) = - \eta(\beta\text{\emph{log}}(R(\psi)^{-1}\mathbf{1}),v) $$
  where the logarithm is defined as in Proposition \ref{grouplikeprimitive} and
  $\beta: A[[\psi]]\rightarrow A[[\kappa_j]]_{j \ge 1}$ is the $A$-linear map that sends each $\psi^j$ to $\kappa_j$. 
 \end{proposition}

\begin{proof}

 Let us first suppose that $\Omega_{g,n}$ is a free boundary CohFT.
 Since $\Omega_{g,n}$ satisfies the axiom involving the forgetful map, we must have in particular, for every $g\ge 1$,
 $$p^*i_{g,1}\widetilde{\Omega}^+(v) = i_{g,2}\widetilde{\Omega}^+(v \cdot R(\psi_2)^{-1}\mathbf{1})$$
 for every $v \in A$. Using the fact that $p^*\kappa_j = \kappa_j - \psi_2^j$ for every $j$, we get 
 $$p^*\text{exp}(\sum_j\phi_j\kappa_j) = \text{exp}(\sum_j\phi_j\kappa_j)\cdot \text{exp}(-\sum_j\phi_j \psi_2^j)$$
 thus for every $v \in A$, we must have
 $$(\widetilde{\Omega}^+\cdot \text{exp}(-\sum_j\phi_j \psi_2^j))(v) =
 \widetilde{\Omega}^+(v \cdot R(\psi_2)^{-1}\mathbf{1})$$
 which immediately implies $R(\psi_2)^{-1}\mathbf{1} = (\text{exp}(-\sum_j\phi_j \psi_2^j))^*$ where the star
 indicates the dual with respect to $\eta$. This is precisely the stated formula.

 Let us define, for the rest of the proof,
 $$\omega_{g,n}(v_1 \otimes \ldots \otimes v_n) = \theta(\alpha^g \cdot v_1 \cdot \ldots \cdot v_n)$$
 for every $v_1, \ldots, v_n \in A$, where $\theta$ is the Frobenius trace of $A$.
 It is easily proved that $\omega_{g,n}$ is a \emph{nodal} CohFT. Notice
 that it is nothing but the degree-zero part of any cohomological field theory, of any type, with base $A$.
 
 Let us now suppose that the formula in the statement of the proposition is satisfied. This formula simply
 says the following: if $\widetilde{\Omega}^+ = \text{exp}(\phi_1 \kappa_1 + \phi_2 \kappa_2 + \ldots)$ 
 (compare with Proposition \ref{howisomegalike}), then 
 $R(\psi)^{-1}\mathbf{1} = \text{exp}(-a_1 \psi - a_2 \psi^2 - \ldots)$
 where the $a_i$ are defined in such a way that,
 for every $w \in A$ and every $i$, $\eta(a_i, w) = \theta(a_i \cdot w) = \phi_i(w)$. Then
 setting $T(z) = z(\mathbf{1} - R(\psi)^{-1}\mathbf{1}) \in z^2 A[[z]]$, Lemma \ref{tremendousexponential} says that
 for every $v_1, \ldots, v_n \in A$, we have
 $$ \theta( \alpha^g \cdot R(\psi_1)^{-1}v_1 \cdot \ldots \cdot R(\psi_n)^{-1}v_n
 \cdot \sum_{m \ge 0} \frac{1}{m!}(p_m)_*(T(\psi_{n+1})\cdots T(\psi_{n+m})))=$$
 $$=i_{g,n}\widetilde{\Omega}^+(\alpha^g \cdot R(\psi_1)^{-1}v_n\cdot \ldots \cdot R(\psi_n)^{-1}v_n).$$
 The left-hand side of the last equality is the restriction of $R\omega$ of Definition \ref{matrixaction} to 
 the smooth parts $\mathcal{M}_{g,n}$ of the moduli spaces. Now, $R\omega$ is a nodal CohFT by Proposition
 \ref{staysfieldtheory} and its restriction to $\mathcal{M}_{g,n}$ is a free boundary CohFT by Proposition 
 \ref{cohernodal}. Thus, the formulas of Theorem \ref{classificationfree} define a free boundary CohFT.
 \end{proof}

 Looking carefully at the last proof, we see that we do not need semisimplicity: the formulas of Theorem
 \ref{classificationfree} will give a free boundary CohFT even if $A$ is not semisimple. However, it is not true
 anymore that every free boundary CohFT is given by those formulas, since the very construction of
 $\widetilde{\Omega}^+$ requires the invertibility of $\alpha$, which is equivalent to semisimplicity by Proposition
 \ref{invertible}.
 
 \section{Classification of nodal theories}

\subsection{Existence of nodal theories}

Up to now, we have classified smooth theories. In particular, fixing a symmetric, commutative, semisimple
Frobenius algebra structure $(A, \cdot, \mathbf{1}, \eta)$, we have found what follows.

\begin{itemize}
\item A fixed boundaries theory $(\widetilde{\Omega}_{g,n})_{g,n}$ is uniquely determined by a sequence
$(\phi_j)_{j \ge 1}$ of elements of $A^*$. 
Setting $\widetilde{\Omega}^+ = \text{exp}(\sum_{j \ge 1}\phi_j\kappa_j):
A \rightarrow \mathbb{C}[\kappa_j]_{j\ge 1}$,
we have $$\widetilde{\Omega}_{g,n}(v_1 \otimes \ldots \otimes v_n) = 
p^*i_g^*\widetilde{\Omega}^+(\alpha^g \cdot v_1 \cdot \ldots \cdot v_n)$$
with notation as in Theorem \ref{classificationfixed}, for $g \ge 1$, while the part $g = 0$ can be obtained from
the part $g = 1$.
\item A free boundaries theory $(\Omega_{g,n})_{g,n}$ is uniquely determined by a fixed boundaries theory
$(\widetilde{\Omega}_{g,n})_{g,n}$, which is its pull-back under the bundle map 
$\pi: \widetilde{\mathcal{M}}_{g,n} \rightarrow \mathcal{M}_{g,n}$, and by and $\text{End}(A)$-valued power series
$R(z)$ that satisfies the symplectic condition and $R(0) = \text{Id}$. In this way we have
$$\Omega_{g,n}(v_1 \otimes \ldots \otimes v_n) =
\widetilde{\Omega}_{g,n}(R(\psi_1)^{-1}v_1 \otimes \ldots \otimes R(\psi_n)^{-1}v_n).$$

The endomorphism $R(z)$ cannot be freely chosen: it must satisfy the condition
$$ \text{log}\widetilde{\Omega}^+(v) = - \eta(\beta\text{log}(R(\psi)^{-1}\mathbf{1}),v) $$
for every $v \in A$, as stated in Proposition \ref{hereisthelogarithm}.
\end{itemize}

We will now deal with nodal theories. When talking about dual graphs, contributions and so on,
we refer to the notations of Section \ref{Givental's group action}.

\begin{proposition}
\label{existence}
Let $(\Omega_{g,n})_{g,n}$ be a free boundaries theory with base $(A,\cdot, \mathbf{1}, \eta)$.
Then there exists a nodal theory $(\overline{\Omega}_{g,n})_{g,n}$ whose restriction to the smooth part
$\mathcal{M}_{g,n}$ of the moduli space $\overline{\mathcal{M}}_{g,n}$ is $(\Omega_{g,n})_{g,n}$.
\end{proposition}

\begin{proof}
Let $R(z)$ be the $\text{End}(A)$-valued power series
determined by the theory $\Omega$. By Corollary \ref{symplecticcondition} and Remark
\ref{tounderstand}, $R(z)$ satisfies the symplectic condition and $R(0) = \text{Id}$; it also satisfy the 
compatibility condition of Proposition \ref{hereisthelogarithm}. Let
$\omega_{g,n} : A^{\otimes n} \rightarrow H^{\bullet}(\overline{\mathcal{M}}_{g,n})$ be defined as
$$\omega_{g,n}(v_1 \otimes \ldots \otimes v_n) = \theta(\alpha^g \cdot v_1 \cdot \ldots \cdot v_n)$$
where $\theta$ is the Frobenius trace of $A$.
Then it is easy to see that $\omega$ is a CohFT. 
Notice that this is simply the degree zero part of $\Omega$. We define, as in the proof of Proposition
\ref{hereisthelogarithm},
$$ \overline{\Omega} = R\omega$$
using the $R$-matrix action of Definition \ref{matrixaction}. By Proposition \ref{staysfieldtheory}, 
$\overline{\Omega}$ is a nodal field theory, and the fact that its restriction to $\mathcal{M}_{g,n}$ is precisely
$\Omega$ is proven in Proposition \ref{hereisthelogarithm}.
\end{proof}

\subsection{Uniqueness for nodal theories}

Now that we have seen how to construct a nodal theory from a smooth one, we
proceed to show uniqueness. Notice that we did not mention
semisimplicity in Proposition \ref{existence}; however, in the following theorem, semisimplicity
is essential.

\begin{theorem}
Let $\Omega$ be a free boundaries CohFT with semisimple base $(A, \cdot, \mathbf{1}, \eta)$.
Then there exists a unique nodal CohFT $\overline{\Omega}$ that has $\Omega$ as its restriction to the smooth
stratum $\mathcal{M}_{g,n}$.
\end{theorem}

We first prove a general lemma which explains how an ambiguity can arise when patching two cohomology classes.

\begin{lemma}
\label{patchingambiguity}
Let $M$ be a smooth complex manifold and
let $j: S \rightarrow M$ 
be the immersion of a closed smooth submanifold of complex codimension $c$. Let
$\nu_S: N \rightarrow S$ be its normal bundle with $N$ the tubular neighbourhood of $S$ given by
the tubular neighbourhood theorem. Let $[\alpha_1], [\alpha_2] \in H^k(M)$
such that $[\alpha_1]|_S = [\alpha_2]|_S$ and $[\alpha_1]|_{S^c} = [\alpha_2]|_{S^c}$. Then
$$ [\alpha_1] - [\alpha_2] \in j_*\text{\emph{Ann}}_{H^{k -2c}(S)}(\text{\emph{c}}(\nu_S)) $$
where $\text{\emph{c}}(\nu_S)$ is the Chern class of the vector bundle $\nu_S$.
\end{lemma}

\begin{proof} The Thom isomorphism theorem says that, for every integer $k$,
$j_*: H^{k-2c}(S) \xrightarrow{\sim} H^k(N, N \setminus S)$,
and the excision property of cohomology implies that $H^k(N,N\setminus S) \simeq 
H^k(M, M \setminus S)$. Notice that excision can be
applied as long as $S$ is a \emph{closed} submanifold.
This allows us to write the cohomology long exact seqence of the immersion 
$M \setminus S \subseteq M$ as
$$ \cdots \rightarrow H^{k-2c}(S) \rightarrow H^k(M) \rightarrow
H^k(M \setminus S) \rightarrow H^{k-2c+1}(S) \rightarrow \cdots $$
 Now, $[\alpha_1] - [\alpha_2] = j_*(b)$ for some 
$b \in H^{k-2c}(S)$ since $[\alpha_1]$ and $[\alpha_2]$ go to same class when pushed to 
$H^k(M \setminus S)$. Similarly, $([\alpha_1] - [\alpha_2])|_S = 0$.
But $j_*(b)|_S = \text{eul}(\nu_S)\cdot b$, therefore we get our result.
\end{proof}

Thus to prove the uniqueness of patchings we have to show that we can always reduce ourselves to the
case in which $\text{c}(\nu_S)$ is not a zero divisor. To do so, we will construct a special stratification
of the moduli space $\overline{\mathcal{M}}_{g,n}$.

\subsubsection{Stratification of $\overline{\mathcal{M}}_{g,n}$}

\begin{definition}
\label{specialcomponents}
Suppose $n > 0$ and let $C$ be a curve in $\overline{\mathcal{M}}_{g,n}$. The \emph{special component}
of $C$ is its irreducible component containing the $n$-th marked point. The connected components
of the curve obtained from $C$ by cancelling out the special component are called \emph{non-special}.
Fixing a curve $C$, the datum of
\begin{itemize}
\item the topological type of the special component of $C$,
\item the number of marked points of the special component $C$,
\item the number of nodes that link the special component of $C$ with the non-special ones,
\end{itemize}
is called the \emph{special type} of $C$.
In other words, two curves in $\overline{\mathcal{M}}_{g,n}$
are said to be of the same special type if their special components
are homeomorphic, have the same number of marked points, and the same number of nodes that link them
to the non-special components.

We denote by $\mathcal{M}_{\tau}^{g,n} \subseteq \overline{\mathcal{M}}_{g,n}$ 
the set of curves with special type $\tau$.
\end{definition}

\begin{remark}
\label{unionofstrata}
Let $C \in \overline{\mathcal{M}}_{g,n}$ be a curve and let $\Gamma_C$ be its dual graph.
Let $v_{sp} \in \Gamma_C$ be the vertex to which the leg labeled $n$ is attached.
Then the special type of $C$ is determined by
the subgraph $\Gamma_C'$ obtained from $C$ by cancelling out all the vertices except $v_{sp}$, and all
the edges except those that have $v_{sp}$ at one end.

Thus, if $C \in S$ where $S$ is an \emph{open} boundary stratum of $\overline{\mathcal{M}}_{g,n}$,
and if $\tau$ is the special type
determined by $\Gamma_C'$, then $S \subseteq \mathcal{M}_{\tau}^{g,n}$.
Indeed, all the curves of an open boundary stratum have the same dual graph.

This shows that each $\mathcal{M}_{\tau}^{g,n}$ is a union of open boundary strata. Since two different open
boundary strata are disjoint, we see at once that if $\tau \neq \sigma$, then 
$\mathcal{M}_{\tau}^{g,n} \cap \mathcal{M}_{\sigma}^{g,n} = \emptyset$. This is just another way to say that
the special type of a curve is well defined.
\end{remark}

\begin{remark}
\label{remarkutile}
Let $\mathcal{M}_{\tau}^{g,n} \subseteq \overline{\mathcal{M}}_{g,n}$ with
$\tau$ a special type represented by a smooth
special component $C_{sp}$ of genus $\gamma$ with $\nu+k$ marked points (the nodes linking $C_{sp}$ to
the non-special components are counted as marked points as well, and they account for the $k$ term
in the previous expression). Let $C \in \mathcal{M}_{\tau}^{g,n}$ and let
$C'$ be the curve obtained from $C$ by eliminating the special component.
Let $C' = C'_1 \sqcup \ldots \sqcup C'_l$ be its decomposition in connected components, with
$C'_i \in \overline{\mathcal{M}}_{g_i, n_i+\mu_i}$ for each $i$, $\mu_i$ being the number of nodes that
$C'_i$ shares with $C_{sp}$. Then we have the relations
\begin{equation}
\label{relations}
\gamma + \sum_{i=1}^l (g_i + \mu_i - 1) =  g \ \ \ , \ \ \ \nu + \sum_{i=1}^ln_i = n \ \ \ \text{and}
\ \ \ \sum_{i=1}^l \mu_i = k
\end{equation}
For $A$ a finite set made of triples $(g',n',\mu')$ of non-negative
integers such that $2g'-2+n'+\mu' > 0$, $g' \le g$, $n' \le n$ and $\mu' \le g + 1$, let us write 
$\mathcal{M}_A = \prod_{(g_i, n_i, \mu_i) \in A} \mathcal{M}_{g_i,n_i+\mu_i}$. Let $X$ be
the family of the $A$'s whose elements satisfy Relations (\ref{relations}). Then we have
$$\mathcal{M}_{\tau}^{g,n} = \bigsqcup_{A \in X} (\mathcal{M}_A \times \mathcal{M}_{\gamma, \nu+k})/F_A$$
where $F_A$ is the (finite) group of automorphisms of some dual graph that depends on $A$ and that
is not of great interest here. Notice that
the sets $\mathcal{M}_A$ are open strata which form, taking the union, a bunch of spaces $\overline{M}_i$,
$i\in I$, that are products of closed moduli spaces.
The $A$'s that contribute to form a certain $\overline{M}_i$ must give the same automorphism group $F_A=F_i$, since
this group is determined by the nodes that the special component shares with the non-special ones, and it does not
depend on the specific topological type of the non-special components.
Thus we can rewrite the previous decomposition as
\begin{equation}
\label{decomposition}
\mathcal{M}_{\tau}^{g,n} = \bigsqcup_{i \in I} (\overline{M}_i \times \mathcal{M}_{\gamma, \nu+k})/F_i 
\end{equation}
Using Relations (\ref{relations}) we see that 
the codimension in $\overline{\mathcal{M}}_{g,n}$ of each component
$(\overline{M}_i \times \mathcal{M}_{\gamma, \nu+k})/F_A$ is $k$, that depends only on
$\tau$. Therefore, in (\ref{decomposition}) we have decomposed $\mathcal{M}_{\tau}^{g,n}$ in connected
components, all with the same dimension.

Let $\tau$ be a special type represented by a non-smooth special component 
$C_{sp} \in \overline{\mathcal{M}}_{\gamma', \nu'+k}$, which must therefore have only non-separating nodes.
If $\mu$ is the number of non-separating nodes of
$C_{sp}$, if $\gamma = \gamma' - \mu$ and $\nu = \nu' + 2\mu$, then the normalization
of the special component lies in $\mathcal{M}_{\gamma,\nu+k}$. The decomposition
(\ref{decomposition}) is still valid, with $\gamma'$, $\nu'$ instead of $\gamma$, $\nu$ in Relations
(\ref{relations}) (but \emph{not} in decomposition (\ref{decomposition})!).
The decomposition is still equidimensional, and the codimension of $\mathcal{M}_{\tau}^{g,n}$ in 
$\overline{\mathcal{M}}_{g,n}$ is now $\mu + k$.
\end{remark}

\subsubsection{Order of the strata}

In all this section we will fix, once and for all, an ambient space $\overline{\mathcal{M}}_{g,n}$ where all the
strata lie. We will omit the indices $g$ and $n$ when it causes no confusion.

\begin{definition}
\label{specialordering}
We say that two special types are such that $\tau > \tau'$ 
(the special type $\tau$ is {\em greater} then $\tau'$) if $\tau \not= \tau'$ 
and at least one point of $\mathcal{M}_{\tau'}$ lies in the closure of~$\mathcal{M}_\tau$,
that is,
$$
\overline{\mathcal{M}}_\tau \cap \mathcal{M}_{\tau'} \not= \emptyset.
$$
\end{definition}
This relation is transitive and therefore determines a partial order on the special types. 

\begin{notation}
Denote by $\mathcal{U}_\tau$ the union $\bigcup_{\tau' \geq \tau} \mathcal{M}_{\tau'}$. 
\end{notation}

\begin{proposition}
The set $\mathcal{U}_\tau$ is open in $\overline{\mathcal{M}}_{g,n}$.
\end{proposition}

\begin{proof} This follows from the definition of the order:
for any $x \in \mathcal{M}_{\tau'}$ and $y$ close enough to $x$,
we have $y \in \mathcal{M}_{\tau''}$ for some $\tau'' \ge \tau'$. \end{proof}

\begin{example}\label{exstrat}

Here is a graphic representation of the stratification we have defined in the case of $\overline{\mathcal{M}}_{1,2}$.
The special strata are four, labeled in the following figure by the letters $A$, $B$, $C$ and $D$. Notice that they
are disjoint (see Remark \ref{unionofstrata}).

\begin{center}
\includegraphics[width=25em]{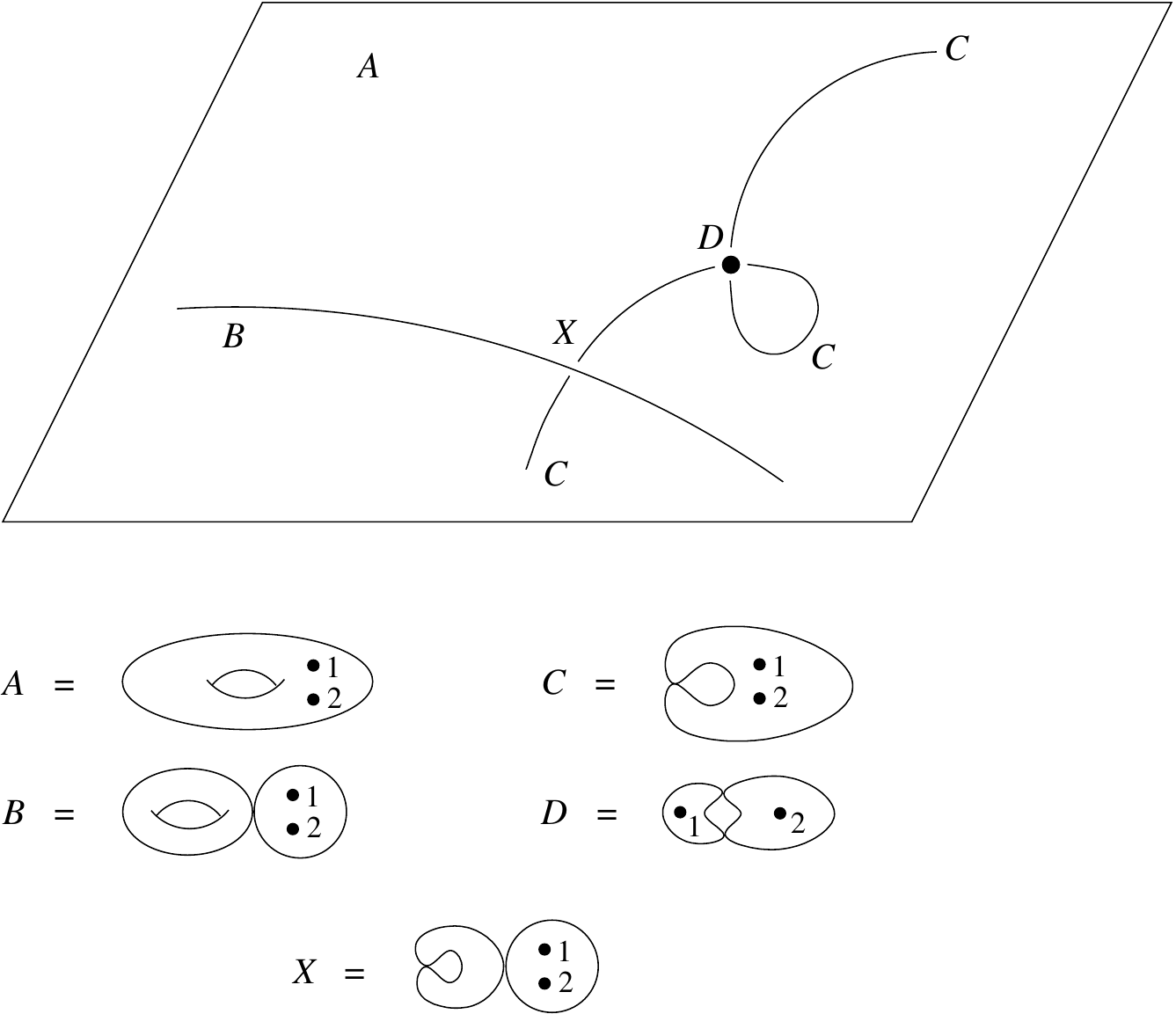}
\end{center}

The partial order of these strata is shown in the diagram, the smaller strata being on the left.
\begin{center}
\includegraphics[width=7em]{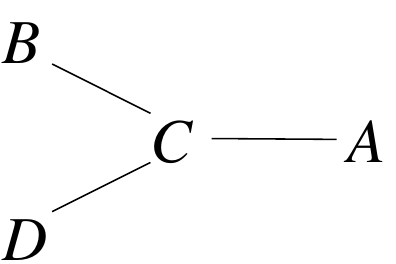}
\end{center}
Note that the order is not compatible with the dimensions of the strata:
for instance, we have $B<C$ even though these cells have the same dimension.
This is due to the fact that $B$ does not lie in the closure of $C$ entirely, but contains one point~$X$ that does.

In the example above we can observe that $A$, $A \cup C$, $A \cup C \cup B$, $A \cup C \cup D$,
and $A \cup B \cup C \cup D$ are open. Moreover
\begin{itemize}
\item $C$ is closed and smooth in $C \cup A$ with a normal bundle of rank~1;
\item $B$ is closed and smooth in $B \cup C \cup A$ with a normal bundle of rank~1;
\item $D$ is closed and smooth in $D \cup C \cup A$ with a normal bundle of rank~2.
\end{itemize}

We will now see that these properties hold in general.

\end{example}

\begin{lemma}
\label{closedunion}
Every $\mathcal{M}_{\tau}$ is a smooth closed sub-orbifold of $\mathcal{U}_\tau$.
\end{lemma}

\begin{proof}
Suppose that $\mathcal{M}_\tau$ has self-intersections, let $\Sigma$ be a point in the 
self-intersection, and let $\sigma\in\text{Aut}(\Sigma)$ be an automorphism that switches two intersecting branches
of $\mathcal{M}_\tau$.
Now, let $V\in T_\Sigma\mathcal{M}_\tau$ be tangent to one of the intersecting branches (therefore it is not
tangent to the other branch, since all the strata have normal crossings). Then $V$ represents a 
deformation of $\Sigma$ which does not change the special component, and a deformation of $\sigma(\Sigma)$ which changes
the special component. Since however any automorphism of $\Sigma$ must preserve its special component, we 
have a contradiction.

Let $x \in \mathcal{U}_\tau \setminus \mathcal{M}_\tau$, so that $x\in \mathcal{M}_{\tau'}$ for some $\tau'>\tau$.
If there were no neighbourhoods $V$ of $x$ with $V \cap \mathcal{M}_\tau = \emptyset$, then 
$x \in \overline{\mathcal{M}_\tau}$, so that $\mathcal{M}_{\tau'}\cap\overline{\mathcal{M}_{\tau}}\neq\emptyset$
and $\tau > \tau'$ by definition of the order. This contradiction shows that 
$\mathcal{M}_\tau \subseteq \mathcal{U}_\tau$ is closed.
\end{proof}

\begin{lemma} 
\label{whatisthenormbundle}
The normal bundle of $\mathcal{M}_{\tau} \subseteq \mathcal{U}_\tau$
is the direct sum of the line bundles $L' \otimes L''$ for $(L',L'')$ tangent pairs at the nodes that link 
the special component to the non-special ones, and at the non-separating nodes of the special component.
\end{lemma}

\begin{proof}
Let us keep the notations of Remark \ref{remarkutile}.
First, let $\tau$ be a special type represented by a smooth special component.
From decomposition (\ref{decomposition}) of Remark \ref{remarkutile},
it suffices to show the statement for each
$(\mathcal{M}_A \times \mathcal{M}_{\gamma, \nu+k})/F_A$. This last object is a
semi-open boundary stratum, which is the image of a sewing map 
$s: \mathcal{M}_A \times \mathcal{M}_{\gamma, \nu+k} \rightarrow \overline{\mathcal{M}}_{g,n}$.
Thus its normal bundle is the sum of the line bundles
$L' \otimes L''$ for tangent pairs $(L', L'')$ at the points sewed together by $s$. Since the special component
is smooth, the lemma is proved for this case.

Let now $\tau$ be a special type represented by a non-smooth special component 
$C_{sp} \in \overline{\mathcal{M}}_{\gamma', \nu'+k}$ whose
normalization lies in $\mathcal{M}_{\gamma,\nu+k}$.
Then by Remark \ref{remarkutile}, each component of the decomposition (\ref{decomposition}) is the image
of a sewing map $s:\mathcal{M}_A \times \mathcal{M}_{\gamma, \nu+k} 
\rightarrow \overline{\mathcal{M}}_{g-\mu,n+2\mu}$ followed by a non-separating sewing map
$q: \overline{\mathcal{M}}_{g-\mu, n+2\mu} \rightarrow \overline{\mathcal{M}}_{g,n}$.
The normal bundle to the image of $q$ is the sum of the line bundles $L' \otimes L''$ for tangent pairs
$(L',L'')$ at the points sewed together by $q$. These account for the non-separating nodes of the special
components, while the other nodes have been dealt with in the smooth case, and the lemma is proved.
\end{proof}

\begin{lemma}
\label{chernnozerodivisor}
Let us consider the decomposition (\ref{decomposition}) of
Remark \ref{remarkutile}:
$$ \mathcal{M}_{\tau} = \bigsqcup_{A \in X} (\mathcal{M}_A \times \mathcal{M}_{\gamma, \nu+k})/F_A.$$
Then, below degree $\gamma/3$, the Chern class of the normal bundle $\nu_{\tau}$ of $\mathcal{M}_{\tau}$ in
$\mathcal{U}_\tau$ is not a zero divisor. In other words, every element of
$\text{\emph{Ann}}_{H^{\bullet - 2k}(\mathcal{M}_{\tau})}(\text{\emph{c}}(\nu_{\tau}))$ has degree strictly
higher than $\gamma/3$.
\end{lemma}

\begin{proof}
Let us write 
$$ \mathcal{M}_{\tau} = \bigsqcup_{A \in X} (\mathcal{M}_A \times \mathcal{M}_{\gamma, \nu+k})/F_A
:= \bigsqcup_{A \in X} T_A$$
with notation as in Remark \ref{remarkutile}. 
 Looijenga's theorem says that, below degree $\gamma/3$ (the so-called \textquotedblleft stable
range \textquotedblright ),
$$H^{\bullet}(\mathcal{M}_{\gamma,\nu+ k}) \simeq H^{\bullet}(\mathcal{M}_{\gamma})[\psi_1, \ldots, \psi_{\nu +k}] .$$
In particular the $\psi$-classes are \emph{free} generators of the stable range of the cohomology of
$\mathcal{M}_{\gamma,\nu+k}$. Now, for each $A \in X$, we have 
$H^{\bullet}(T_A) \subseteq H^{\bullet}(\mathcal{M}_A) \otimes H^{\bullet}(\mathcal{M}_{\gamma,\nu+k})$,
and
$$\text{c}(\nu_{T_A}) = - \sum_{i=1}^k (\psi'_i \otimes 1 + 1 \otimes \psi_{\nu+i}) \in H^{\bullet}(T_A)$$
where for each $i$, $\psi'_i$ is the $\psi$-class on $\mathcal{M}_A$ at the point sewed to the point
marked $\nu+i$ on $\mathcal{M}_{\gamma, \nu+k}$. This implies that each $\text{c}(\nu_{T_A})$ is not
a zero-divisor below degree $\gamma/3$. By Lemma \ref{whatisthenormbundle}, we have
$\text{c}(\nu_{\tau}) = \sum_{A \in X} \text{c}(\nu_{T_A})$, and this is in fact a direct sum. Therefore
we conclude that $\text{c}(\nu_{\tau})$ is not a zero-divisor below degree $\gamma/3$, which is what we had to prove.
\end{proof}

We will need one more technical lemma.

\begin{lemma}\label{patchingstrata}
 Let $\tau_0$ be a special stratum and suppose that the restriction of $\overline{\Omega}_{g,n}$ on each stratum 
 $\mathcal{M}_\tau \subseteq \overline{\mathcal{M}}_{g,n}$ with $\tau \ge \tau_0$ is known. Then the restriction of
$\overline{\Omega}_{g,n}$ to each $\mathcal{U}_\tau$ with $\tau \ge \tau_0$
is known below degree $\gamma/3$, where 
$\gamma$ is the genus of the special component corresponding to the stratum $\tau$.
\end{lemma}

\begin{proof}
 We reason by descending 
induction using the order of Definition \ref{specialordering} on the special strata. If $\tau$ is the biggest special
type, then $\mathcal{U}_\tau = \mathcal{M}_\tau$, where $\overline{\Omega}_{g,n}$ is determined by hypothesis.
Now if we have determined the theory on each $\mathcal{U}_\tau$ for $\tau$ in some union of intervals
$I$ containing the biggest special type,
choose a maximal $\tau' \notin I$
and consider $\mathcal{U}_{\tau'}$. Then, by induction, the theory is determined on 
$\mathcal{U}_{\tau'}\setminus\mathcal{M}_{\tau'}$ up to degree $\gamma'/3$ (notice that if $\tau'\le \tau$ then 
$\gamma'\le \gamma$).
By Lemma \ref{closedunion}, $\mathcal{M}_{\tau'}$ is a smooth 
closed sub-orbifold of $\mathcal{U}_{\tau'}$. Then we can patch, uniquely below degree $\gamma'/3$,
the theories on $\mathcal{M}_{\tau'}$ and on $\mathcal{U}_{\tau'}\setminus\mathcal{M}_{\tau'}$
thanks to Lemmas \ref{patchingambiguity} and \ref{chernnozerodivisor}.
\end{proof}

Now we are ready to prove the \textquotedblleft uniqueness\textquotedblright \ 
part of the classification of nodal theories. We keep the
notation of the beginning of this section.

\begin{theorem}
A nodal CohFT $\overline{\Omega}$ is uniquely determined by its restriction $\Omega$ to the smooth part
$\mathcal{M}_{g,n}$. 
\end{theorem}
 
 \begin{proof}
By definition, 
$$\overline{\Omega}_{0,3}(v_1 \otimes v_2 \otimes v_3) = 
\eta(v_1 \cdot v_2, v_3) = \Omega_{0,3}(v_1 \otimes v_2 \otimes v_3)$$ thus
$\overline{\Omega}_{0,3}$ is uniquely determined by $\Omega_{0,3}$.

Let $d \ge 1$, let us suppose we know
$\overline{\Omega}_{g',n'}$ for every $(g',n')$ such that $3g'-3+n' < d$, and let 
$3g-3+n = d$.  We have to show that we can determine $\overline{\Omega}_{g,n}$.

Let $G > 9g-9+3n$, and let us consider the sewing map
$$s: \overline{\mathcal{M}}_{g,n} \times \mathcal{M}_{G,2} 
\rightarrow \overline{\mathcal{M}}_{g+G,n} $$
that identifies the points labeled $n$ and $2$ respectively (notice the lack of the overline in the second
space: we only want to consider sewings of smooth surfaces). Therefore, the point marked $1$ on
$\mathcal{M}_{G,2}$ goes to the point marked $n$ on $\overline{\mathcal{M}}_{g+G,n}$.
Let $N$ be the usual tubular neighbourhood
of the boundary stratum $S$ in which the image of this map lies. 
Let $I$ be the set of the $\tau$'s such 
that $\mathcal{M}^{g+G,n}_{\tau} \cap \partial N \neq \emptyset$, and let
$$\mathcal{U} = \bigcup_{\tau \in I} \mathcal{M}^{g+G,n}_{\tau}$$
This accounts for a finite stratification of $\mathcal{U}$ in which every stratum has special component 
of genus $G$ or higher. Here is a graphic representation of the strata of $\mathcal{U}$ in the case $g=1$, $n=2$.
See also Example \ref{exstrat} for more comments.

\begin{center}
 \includegraphics[width=25em]{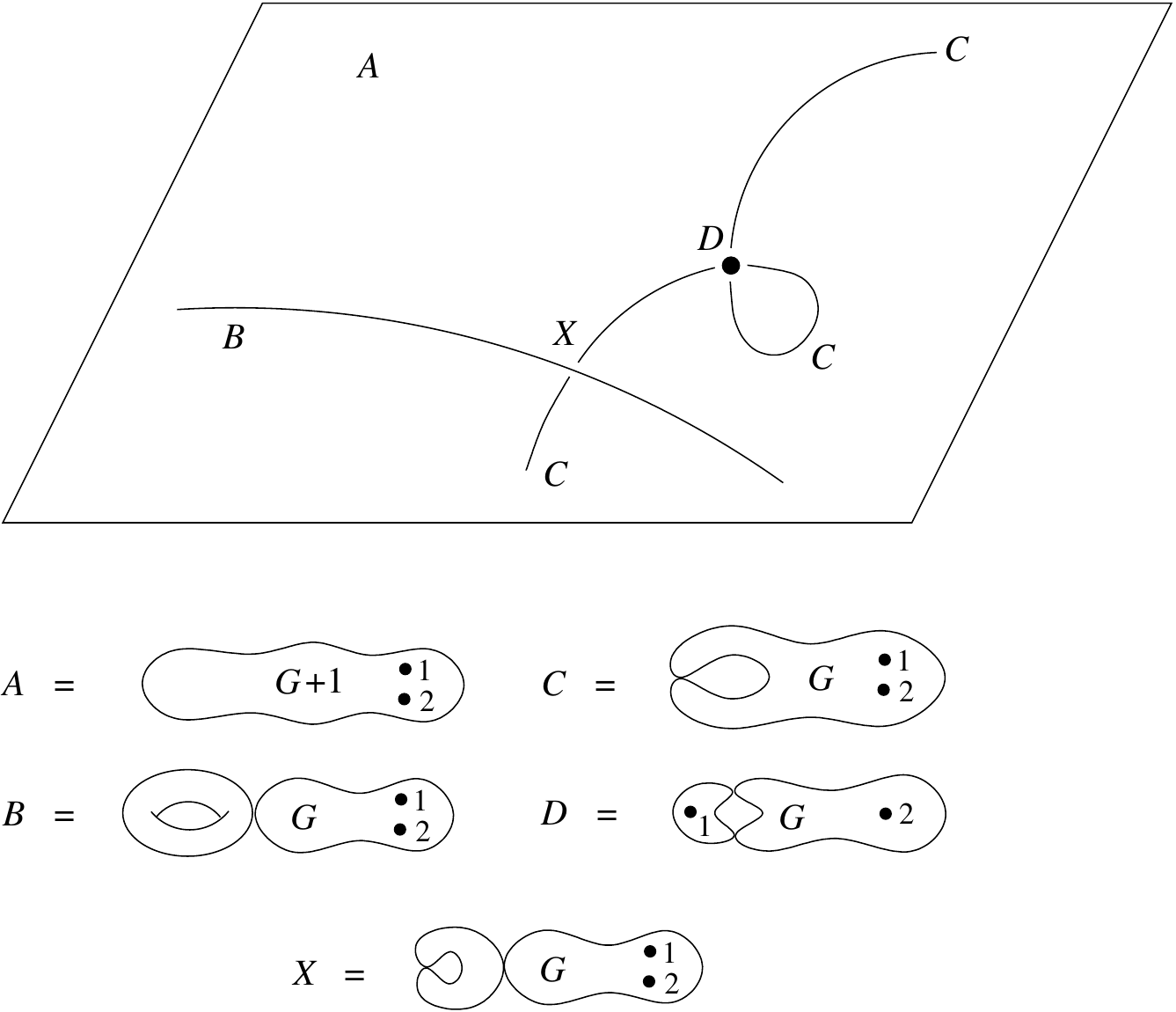}
\end{center}

Now every special stratum $\mathcal{M}^{g+G,n}_\tau$ with $\tau \in I$ is a product of a smooth moduli space
and some $\overline{\mathcal{M}}_{g',n'}$'s with $3g'-3+n'<d$ (see Remark \ref{remarkutile}).
Thus our inductive hypothesis together with the sewing axiom implies that
the restriction of $\overline{\Omega}_{g+G,n}$ to each $\mathcal{M}^{g+G,n}_\tau$ is uniquely determined.

Let $\Omega_{\mathcal{U}}$ be the restriction of $\overline{\Omega}_{g+G,n}$ to $\mathcal{U}$. Notice that 
$\mathcal{U}$ is contained in $\mathcal{U}_{\tau_0}$ where $\tau_0$ is the special stratum 
(in $\overline{\mathcal{M}}_{g+G,n}$) corresponding to curves whose special component is a smooth curve of genus $G$
with one marked point and one node attaching it to the other components.
By Lemma \ref{patchingstrata}, $\Omega_\mathcal{U}$ is uniquely determined below degree $G/3$. Since 
$\partial N \subseteq \mathcal{U}$, we see that the restriction of $\overline{\Omega}_{g+G,n}$ to $\partial N$ is
uniquely determined below degree $G/3$.

Now the sewing axiom yields $$\overline{\Omega}_{g+G,n}(v_1\otimes \ldots \otimes v_n)|_{\partial N} =
\nu^*\eta^{\mu\nu}\overline{\Omega}_{g,n}(v_1\otimes \ldots \otimes e_{\mu})\times \Omega_{G,2}(v_n,e_{\nu})$$ and the Gysin
sequence for the circle bundle $\partial N \rightarrow S$ implies that $\nu^*$ induces an isomorphism below 
degree $G/3$
$$H^\bullet(\partial N) \simeq H^\bullet(\overline{\mathcal{M}}_{g,n})\otimes H^\bullet(\mathcal{M}_G)[\psi_{1}']$$ 
where $\psi_1'$ is the $\psi$-class attached to the first marked point in $\mathcal{M}_{G,2}$. We can conclude that 
$$\eta^{\mu\nu}\overline{\Omega}_{g,n}(v_1\otimes \ldots \otimes e_{\mu})\times \Omega_{G,2}(v_n,e_{\nu})$$
is known below degree $G/3$. Restricting to any fixed smooth surface $\Sigma \in \mathcal{M}_{G,2}$ gives 
$$\overline{\Omega}_{g,n}(v_1\otimes\ldots \otimes (\alpha^G\cdot v_n)),$$ (see the proofs of
Propositions \ref{sewfixedtorus}-\ref{sewnodaltorus}) and since in a semisimple theory 
$\alpha$ is invertible, we conclude that $\overline{\Omega}_{g,n}$ is uniquely determined below degree 
$G/3 > 3g-3+n = \text{dim}\overline{\mathcal{M}}_{g,n}$. Thus $\overline{\Omega}_{g,n}$ is totally 
determined: this completes the inductive step and the proof.

\end{proof}

\end{document}